\documentclass[12pt]{amsart}

\usepackage{mathtools}
\usepackage{titlesec,caption, subcaption}
\usepackage{tikz-cd}
\usepackage{amsmath,amsthm,amsfonts,amscd,amssymb,amsopn,enumerate,enumitem,amssymb}
\usepackage{mathrsfs,xcolor,hyperref,stmaryrd}
\usepackage{wasysym}
\usepackage[all]{xy}
\usepackage{fullpage}
\usepackage{verbatim}
\usepackage{colonequals}
\usepackage[new]{old-arrows}
\hypersetup{hidelinks}
\titleformat{name=\section}{}{\thetitle.}{0.8em}{\centering\scshape}
\titleformat{name=\subsection}[runin]{}{\thetitle.}{0.8em}{\bfseries}[.]
\titleformat{name=\subsubsection}[runin]{}{\thetitle.}{0.8em}{\bfseries}[.]

\theoremstyle{plain}
\newtheorem{thm}{Theorem}[subsection]
\newtheorem*{thm*}{Theorem}
\newtheorem{cor}[thm]{Corollary}

\newtheorem{lem}[thm]{Lemma}

\newtheorem*{conj*}{Conjecture}
\newtheorem*{rem*}{Remark}

\newtheorem*{fact*}{Fact}
\newtheorem*{prop*}{Proposition}

\theoremstyle{definition}
\newtheorem{defn}[thm]{Definition}
\newtheorem*{defn*}{Definition}

\newtheorem*{ack}{Acknowledgement}

\newtheorem{rem}[thm]{Remark}

\newtheorem{prop}[thm]{Proposition}

\newtheorem{lemma}[thm]{Lemma}
\theoremstyle{remark}

\begin{document}

\title{On the pro-modularity in the residually reducible case for some totally real fields}
\author{Xinyao Zhang}

\begin{abstract}
	Let $p$ be an odd prime and $F$ an abelian totally real field of even degree
	in which $p$ splits completely.  We study the fixed-determinant
	pseudo-deformation ring of \mbox{$\mathbf{1}+\bar\chi$}.  Under local
	hypotheses and a degree bound in terms of the tame auxiliary places, we prove
	that every
	irreducible component of at least the expected dimension
	\mbox{$1+2[F:\mathbb{Q}]$} is pro-modular---its generic
	pseudo-representation occurs in a big Hecke algebra---and has exactly this
	dimension.  Under a further cohomological bound, the deformation ring of every
	non-split residual extension is a local complete intersection and all its
	primes are pro-modular.  After a finite abelian totally real base change, this
	componentwise result may be viewed as a potential big $R=\mathbb{T}$ theorem
	at the level of irreducible components.  The proof uses a specialization
	argument that
	constructs a one-dimensional patching prime while preserving, for a single
	stable lattice, both non-split residual reduction and the prescribed local
	conditions.
\end{abstract}

\maketitle

\section{Introduction}

The deformation theory of Galois representations provides a geometric way to
organize congruences between modular forms.  Starting from a residual
representation, one constructs a universal deformation ring whose points
parametrize its $p$-adic lifts.  On the automorphic side, a big Hecke algebra
records systems of Hecke eigenvalues occurring in modular forms of varying
weights.  A big $R=\mathbb{T}$ theorem identifies these two objects and, in
particular, says that the corresponding deformation space is filled, in a
scheme-theoretic sense, by modular points.  This circle of ideas originates in
Mazur's deformation theory \cite{mazur1989deforming} and in the work of
Gouv\^ea and Mazur on the infinite fern
\cite{gouveaMazur1998density}.

When the residual representation is absolutely irreducible, the deformation
functor is represented by a usual deformation ring, and big $R=\mathbb{T}$
theorems can be approached by placing suitably regular automorphic points on
its irreducible components.  This strategy was developed by B\"ockle over
$\mathbb{Q}$ \cite{bockle2001density} and by Allen in polarized and Hilbert
modular settings \cite{allen2019automorphic}.  The residually reducible case is
more subtle.  A semisimple residual representation
$\mathbf{1}\oplus\bar\chi$ need not have a universal deformation ring in the
usual sense, and its natural parameter space is instead the universal
pseudo-deformation space of the residual pseudo-representation
$\mathbf{1}+\bar\chi$; we refer to \cite{bellaiche2009families} for the
general theory.  Moreover, this space remembers all non-split residual
extensions with semisimplification $\mathbf{1}\oplus\bar\chi$ at once, and its
irreducible components cannot in general be studied through the Galois
cohomology of a single residual representation.

In this paper, a prime of a pseudo-deformation ring is called
\textit{pro-modular} if the corresponding pseudo-representation occurs in a
suitable localized big Hecke algebra.  An irreducible component is
pro-modular if its generic point is pro-modular.  Thus, componentwise
pro-modularity is the reduced, componentwise part of a big $R=\mathbb{T}$
statement: it determines which Galois-theoretic components are seen by
completed automorphic forms, without making an assertion about nilpotent
structures.  The purpose of this article is to prove such a componentwise
statement in a residually reducible setting over certain totally real fields
and to develop its consequences.  These include a potential big
$R=\mathbb{T}$ statement at the level of irreducible components and structural
results for deformation rings of non-split residual extensions.

\subsection*{Motivation}

Deo studied the corresponding problem over $\mathbb{Q}$ in
\cite{deo2023density}.  To recall one form of his result, let
\[
 \bar\rho_0=\bar\chi_1\oplus\bar\chi_2:
 G_{\mathbb{Q},Np}\longrightarrow\operatorname{GL}_2(\mathbb{F}),
 \qquad \bar\chi=\bar\chi_1\bar\chi_2^{-1},
\]
and let $R_{\bar\rho_0}^{\mathrm{pd}}$ be the universal
pseudo-deformation ring.  If
$0\ne x\in H^1(G_{\mathbb{Q},Np},\mathbb{F}(\bar\chi))$, let
$\bar\rho_x$ be the associated non-split extension, well-defined up to
equivalence, and let $R_{\bar\rho_x}$ be its universal deformation ring.

\begin{thm}[Deo {\cite[Theorem A and Corollary 3.2.2]{deo2023density}}]
	Let $N_0$ be the tame Artin conductor of $\bar\rho_0$.  Suppose that
	$N_0\mid N$, that
	$\bar\chi|_{G_{\mathbb{Q}_p}}\notin
	\{\mathbf{1},\omega_p^{-1}\}$, that
	\[
	 \dim_{\mathbb{F}}H^1(G_{\mathbb{Q},Np},
	 \mathbb{F}(\bar\chi))=1,
	\]
	and that $p\nmid\varphi(N)$.  Then:
	\begin{enumerate}
		\item the natural morphism to the localized big Hecke algebra
		induces an isomorphism
		\[
		 (R_{\bar\rho_0}^{\mathrm{pd}})^{\mathrm{red}}
		 \xrightarrow{\ \sim\ }\mathbb{T}(N)_{\bar\rho_0};
		\]
		both sides are reduced local complete intersection rings of Krull
		dimension $4$;
		\item for every non-zero class $x$ as above,
		$R_{\bar\rho_x}$ is a local complete intersection ring of Krull
		dimension $4$.
	\end{enumerate}
\end{thm}

Deo's theorem is the motivating model for this paper: it suggests asking which
part of the residually reducible big $R=\mathbb{T}$ phenomenon, and which of its
structural consequences, persist over totally real fields of higher degree.
It is not, however, an input to our proof.  The one-dimensionality of the
relevant extension space is special to the setting over $\mathbb{Q}$ and
generally fails as the degree grows, so the route through the deformation ring
of a single non-split residual extension is unavailable.  We instead formulate
the problem on the irreducible components of the entire pseudo-deformation
space.  We first state the main theorem and its applications; the obstruction
to Deo's method and the relation with Pan's argument are explained afterwards.

\subsection*{Main results and applications}

We now give the precise setting.  Let $p\geq3$, and let $F$ be an abelian
totally real field of even degree in which $p$ splits completely.  Let
$\Sigma$ be a finite set of finite places of $F$ containing the set
$\Sigma_p$ of places above $p$, and let $F_{\Sigma}$ be the maximal extension
of $F$ unramified outside $\Sigma$ and the infinite places.  Let
$\mathcal{O}$ be the ring of integers of a finite extension of
$\mathbb{Q}_p$, with uniformizer $\pi$ and residue field $\mathbb{F}$.  We
fix a continuous totally odd character
\[
 \chi:\operatorname{Gal}(F_{\Sigma}/F)
 \longrightarrow\mathcal{O}^{\times}
\]
and write $\bar\chi$ for its reduction modulo $\pi$.  Put
\[
 d=[F:\mathbb{Q}],\qquad s=|\Sigma\setminus\Sigma_p|,\qquad
 m=\dim_{\mathbb{F}}H^1(F_{\Sigma}/F,
 \mathbb{F}(\bar\chi^{-1})).
\]
For a non-zero class
$x\in H^1(F_{\Sigma}/F,\mathbb{F}(\bar\chi^{-1}))$, let $\rho_x$ be the
non-split extension, unique up to equivalence, of the form
\[
 \rho_x=
 \begin{pmatrix}
  \mathbf{1}&*\\0&\bar\chi
 \end{pmatrix}.
\]
Let $R_x^{\mathrm{aux}}$ be its universal deformation ring with determinant
$\chi$, unramified outside $\Sigma$, and let
$R_{\mathrm{aux}}^{\mathrm{ps}}$ be the corresponding universal
pseudo-deformation ring of $\mathbf{1}+\bar\chi$.  The precise definitions of
the auxiliary deformation problems and of pro-modularity are given in
Definitions~\ref{auxiliary}, \ref{nice}, and~\ref{pro-modular}.

\begin{thm}[Theorems \ref{pro}, \ref{MINIMAL}, and \ref{special case}]
\label{main}
	Assume that $p\mid\operatorname{Nm}(v)-1$ for every
	$v\in\Sigma\setminus\Sigma_p$, and that:
	\begin{enumerate}
		\item $\chi$ is unramified outside $\Sigma_p$;
		\item $\chi(\operatorname{Frob}_v)\equiv1\pmod{\pi}$ for every
		$v\in\Sigma\setminus\Sigma_p$;
		\item $\bar\chi$ extends to a character of $G_{\mathbb{Q}}$;
		\item $\bar\chi|_{G_{F_v}}\notin
		\{\mathbf{1},\omega_p^{-1}\}$ for every $v\mid p$;
		\item $\chi|_{G_{F_v}}$ is de Rham for every $v\mid p$.
	\end{enumerate}
	Then the following assertions hold.
	\begin{enumerate}
		\item If $d\geq7s+4$, every irreducible component of
		$R_{\mathrm{aux}}^{\mathrm{ps}}$ of dimension at least $1+2d$
		is pro-modular and has dimension exactly $1+2d$.
		\item If $d\geq\max\{1+\frac{1}{2}m,7s+4\}$, then, for every
		non-zero class
		$x\in H^1(F_{\Sigma}/F,\mathbb{F}(\bar\chi^{-1}))$, every prime
		ideal of $R_x^{\mathrm{aux}}$ is pro-modular.  Moreover,
		$R_x^{\mathrm{aux}}$ is a local complete intersection ring of
		Krull dimension $1+2d$.
	\end{enumerate}
\end{thm}

The number $1+2d$ is the expected dimension for the fixed-determinant
deformation problem.  For $R_x^{\mathrm{aux}}$, the global Euler
characteristic formula gives this number as a lower bound for every
irreducible component.  No equally direct tangent--obstruction calculation is
available for the pseudo-deformation ring.  This is why part~(1) is stated for
components already known to have dimension at least $1+2d$.  The theorem
determines their dimension and proves that they are automorphic components; it
does not assert that smaller components of the pseudo-deformation ring cannot
occur.

We briefly record the role of the main hypotheses.  The assumptions that
$F/\mathbb{Q}$ is abelian and that $\bar\chi$ extends to
$G_{\mathbb{Q}}$ enter chiefly through the ordinary Fontaine--Mazur input of
Skinner--Wiles and Pan.  The soluble totally real base changes in their
arguments use Washington's theorem to control $p$-primary class groups in an
auxiliary cyclotomic $\mathbb{Z}_{\ell}$-tower.  Total oddness and
extendability make the relevant residual splitting field abelian CM; no
comparable non-abelian input is known
\cite{washington1978class,skinner1999residually} (see also
\cite[Sections~5.7.2--5.7.4]{pan2022fontaine}).  Abelianity is also used
through Leopoldt's conjecture in the bounds for the reducible and dihedral
loci.  Even degree permits the use of a totally definite quaternion algebra;
complete splitting of $p$, together with extendability, makes the local
residual types and codimension counts uniform.  The
constant $7s+4$ records the successive specializations at the places in
$\Sigma\setminus\Sigma_p$ and is not expected to be optimal.  In the generic
local case treated in Subsection~4.4, the argument only needs $7s+2$.  The
unified statement above also includes the remaining local type
$\bar\chi|_{G_{F_v}}=\omega_p$, when $p>3$, which is treated separately in
Subsection~4.5 and accounts for the slightly larger bound.

The componentwise formulation is particularly useful because it has the
following potential form.  In the precise sense that, after one base change,
every irreducible component of at least the expected dimension is seen on the
automorphic side, the result may be viewed as a potential big $R=\mathbb{T}$
theorem at the level of irreducible components.  It does not assert equality of
the full scheme structures.  Requiring the field produced below to be abelian
over $\mathbb{Q}$ keeps the base-changed problem within the scope of this
ordinary Fontaine--Mazur input.

\begin{thm}[Potential pro-modularity]\label{potential-pro-modularity}
	Assume that $\bar\chi$ extends to $G_{\mathbb{Q}}$, that
	$\bar\chi|_{G_{F_v}}\notin\{\mathbf{1},\omega_p^{-1}\}$ for every
	$v\mid p$, and that $\chi|_{G_{F_v}}$ is de Rham for every
	$v\mid p$.  Then there exists a finite extension $F_1/F$ such that
	$F_1/\mathbb{Q}$ is abelian and totally real of even degree and $p$
	splits completely in $F_1$, with the following property.  Let
	$R_{F_1}^{\mathrm{ps},\chi}$ be the universal pseudo-deformation ring
	parametrizing pseudo-deformations of
	$\mathbf{1}+\bar\chi|_{G_{F_1}}$, with determinant
	$\chi|_{G_{F_1}}$, unramified outside the places above $\Sigma$.
	Every irreducible component of $R_{F_1}^{\mathrm{ps},\chi}$ of
	dimension at least $1+2[F_1:\mathbb{Q}]$ is pro-modular and has
	dimension exactly $1+2[F_1:\mathbb{Q}]$.
\end{thm}

\begin{proof}
	We first reduce to the case in which the chosen lift $\chi$ is the
	restriction of a character of $G_{\mathbb{Q}}$.  Since $\chi$ is de Rham
	at every place above $p$, global class field theory and the known
	Leopoldt conjecture for the abelian totally real field $F$ allow us to
	write $\chi=\psi\epsilon^k$, where $\psi$ is a finite order character and
	$k\in\mathbb{Z}$.  Choose an extension $\bar\chi_{\mathbb{Q}}$ of
	$\bar\chi$ to $G_{\mathbb{Q}}$, and put
	\[
	 \bar\psi_{\mathbb{Q}}=\bar\chi_{\mathbb{Q}}\omega_p^{-k},
	 \qquad
	 \chi_{\mathbb{Q}}
	 =\widetilde{\bar\psi}_{\mathbb{Q}}\epsilon^k,
	\]
	where $\widetilde{\bar\psi}_{\mathbb{Q}}$ denotes the Teichm\"uller lift
	of $\bar\psi_{\mathbb{Q}}$.  Write
	$\chi_0=\chi_{\mathbb{Q}}|_{G_F}$.  The character
	$\chi_0\chi^{-1}$ is finite order and reduces to the trivial character.
	Since $p$ is odd, after enlarging $\mathcal{O}$ if necessary, there is a
	finite order character
	$\varphi:G_{F,\Sigma}\to\mathcal{O}^{\times}$, with
	$\bar\varphi=\mathbf{1}$, such that
	$\varphi^2=\chi_0\chi^{-1}$.
	
	Twisting by $\varphi$ identifies the fixed-determinant
	pseudo-deformation rings for $\chi$ and $\chi_0$.  On the automorphic
	side, the same twist identifies the corresponding big Hecke algebras
	after replacing the tame level by a sufficiently small open compact
	subgroup if necessary; see
	\cite[Section 3.2.2]{pan2022fontaine}.  Thus twisting preserves
	pro-modularity as well as the dimensions of irreducible components.  We
	may therefore assume that $\chi$ extends
	to $G_{\mathbb{Q}}$.
	
	By the Grunwald--Wang theorem
	\cite[Theorem 5, p.~80]{artin}, one may choose $F_1$ so that, in
	addition to the properties in the statement, all the auxiliary conditions
	of Theorem~\ref{main} hold after restriction to $G_{F_1}$ and its degree
	satisfies the required inequality.  The extension can be chosen to be
	trivial at the places above $p$, so the local hypotheses there are
	unchanged.  The assertion is therefore Theorem~\ref{main}(1) applied over
	$F_1$.
\end{proof}

\begin{rem}[Relation with the Fontaine--Mazur conjecture]
	Pan describes his proof of the residually reducible Fontaine--Mazur
	conjecture as having two broad stages: a big $R=\mathbb{T}$ input for
	completed cohomology and a classicality theorem; see
	\cite[p.~1032]{pan2022fontaine}.  More precisely, Pan starts with the point
	defined by a given representation, chooses a sufficiently large irreducible
	component through that point, and makes this selected component automorphic
	after a soluble base change; potentially nice primes are sufficient for this
	purpose \cite[Sections~7.2 and~7.4]{pan2022fontaine}.  His application does
	not require a simultaneous statement for all components of a fixed
	pseudo-deformation space.

	The quantifier in Theorem~\ref{main}(1) is different: over the fixed field
	$F$, every irreducible component of the pseudo-deformation ring having at
	least the expected dimension is pro-modular.  Theorem
	\ref{potential-pro-modularity} makes this simultaneous statement after a
	single abelian base change.  This is the sense in which the present result is
	a potential big $R=\mathbb{T}$ theorem; no equality involving the nilpotent
	structures of the two rings is claimed.
	
	The hypothesis in Theorem~\ref{main} deliberately excludes the
	exceptional local type $p=3$ and
	$\bar\chi|_{G_{F_v}}=\omega_3^{-1}=\omega_3$.  Thus the present theorem
	does not by itself prove that missing case.  In \cite{ZHANG2024}, the
	propagation mechanism of Proposition~\ref{promodular to nice} and the
	specialization method of Corollary~\ref{innovation'} are combined with
	additional local--global compatibility and local deformation-theoretic
	inputs at $p=3$.  This yields the corresponding potential big
	$R=\mathbb{T}$ statement in the exceptional case and completes the regular
	two-dimensional Fontaine--Mazur conjecture over $\mathbb{Q}$ for odd
	primes.  The distinction is useful: the present paper supplies the
	componentwise geometric mechanism, while the exceptional local input
	belongs to the later work. See \cite{ZHANG2024} for more details.
\end{rem}

\begin{rem}[A dimension-theoretic application]
	The potential pro-modularity result above also has useful
	dimension-theoretic applications.  In \cite{ZhaZariskiDensity}, it is used
	to prove local-to-global finiteness results for pseudo-deformation rings;
	these results are then applied to study the dimensions of global
	pseudo-deformation rings over $\mathbb{Q}$.  We refer to that work for the
	precise statements and further applications.
\end{rem}

\subsection*{The obstruction to Deo's method and the proof strategy}

We now explain why Deo's result serves as motivation rather than as a method
of proof.  Its one-dimensionality hypothesis permits one to pass from the
pseudo-deformation ring to the deformation ring of a chosen non-split residual
extension and then to use generalized matrix algebra arguments.  Over a
totally real field of higher degree, this hypothesis is too restrictive.  In
fact, if $F$ is totally real of degree $d$ and $M$ is one of the relevant
non-trivial, one-dimensional, totally odd
$\mathbb{F}[G_{F,\Sigma}]$-modules, the global Euler--Poincar\'e
characteristic formula gives
\[
 \dim_{\mathbb{F}}H^1(F_{\Sigma}/F,M)
 -\dim_{\mathbb{F}}H^2(F_{\Sigma}/F,M)=d.
\]
Thus the corresponding extension space has dimension at least $d$.  Choosing
a single extension class no longer captures the geometry of the entire
pseudo-deformation space, so the reduction used over $\mathbb{Q}$ is not
available.  Our proof does not attempt to adapt that reduction; it works with
the irreducible components of the pseudo-deformation space from the outset.

The proof follows instead the overall architecture of Pan's argument for the
Fontaine--Mazur conjecture in the residually reducible case
\cite{pan2022fontaine}, which in turn builds on the ordinary modularity method
of Skinner--Wiles \cite{skinner1999residually}.  Two inputs are central.  First,
ordinary finiteness produces many pro-modular points on the ordinary locus.
Second, a patching theorem identifies the deformation and Hecke spaces locally
at a one-dimensional \textit{nice prime}.  Such a prime is required to be
irreducible, to admit a stable lattice with non-split residual reduction, to
satisfy a Taylor--Wiles type non-dihedral condition, and to be a trivial lift
of its reduction at the places in $\Sigma\setminus\Sigma_p$.  Once an
irreducible component contains a nice prime, the patching theorem makes that
component pro-modular.

Pan's pointwise Fontaine--Mazur application allows him to work with a selected
component after a soluble base change, for which a potentially nice prime is
enough.  To prove Theorem~\ref{main}(1), one must instead construct a nice
prime on an arbitrary component of the expected size over the fixed field
$F$.  All the defining conditions must moreover be compatible with one choice
of stable lattice.  In particular, choosing a lattice with non-split reduction
and subsequently changing it to impose local triviality need not preserve the
first condition.  Resolving this incompatibility is the main technical point
of the present paper.

The new ingredient is Corollary~\ref{innovation'}, a specialization lemma for
two-dimensional pseudo-representations.  With an explicitly bounded loss of
dimension, it makes the finitely many off-diagonal local data at the auxiliary
places compatible.  Applied in both off-diagonal directions and combined with
a final one-dimensional specialization, it allows us to choose a single
stable lattice whose residual representation is non-split and whose
restrictions at the places in $\Sigma\setminus\Sigma_p$ are trivial lifts of
their reductions.  This simultaneous lattice control is precisely what is
needed to construct a nice prime on an arbitrary component of the expected
size.  The lemma is independent of the automorphic input and is also used in
the more delicate $p=3$ application discussed above.

We finish with an outline of the proof of Theorem~\ref{main}.  If $C$ is a
component as in part~(1), its ordinary locus has dimension at least $1+d$.
Pan's ordinary finiteness theorem shows that a component of this locus maps
finitely and surjectively to the Iwasawa algebra, so regular de Rham
automorphic points are dense.  We next specialize inside this pro-modular
ordinary component.  Imposing the diagonal local conditions at the $s$ places
away from $p$, and then applying Corollary~\ref{innovation'} to rows and
columns, costs at most $7s$ dimensions.  The remaining dimension produces a
one-dimensional irreducible prime, and the simultaneous lattice construction
described above shows that it is nice.  Pan's patching theorem at a nice prime,
recalled as Theorem~\ref{R=T}, makes $C$ pro-modular.  Comparing the lower
bound for the ordinary locus with the upper bound from ordinary finiteness
gives $\dim C=1+2d$.

Part~(2) follows by comparing the deformation ring of a fixed non-split
extension with the pseudo-deformation ring.  The degree hypothesis makes the
reducible locus too small to contain a component of $R_x^{\mathrm{aux}}$; its
generic representation is therefore irreducible and lies on a component
covered by part~(1).  The resulting exact dimension, together with the
standard tangent--obstruction presentation, shows that the defining equations
form a regular sequence.  This gives the asserted local complete intersection
property.  Thus Deo's theorem supplies the motivating structural outcome,
Pan's work supplies the global and patching architecture, and the
specialization argument makes that architecture compatible with the
componentwise statement proved here.

\subsection*{Organization and audience}

The paper is intended for readers familiar with the basic language of Galois
deformations and Hilbert modular forms, but not necessarily with all the
details of Pan's proof.  Section~2 recalls two-dimensional
pseudo-representations and proves the specialization result,
Corollary~\ref{innovation'}, in a form that can be reused independently of the
automorphic application.  Section~3 studies the deformation rings attached to
fixed extension classes, the reducible and dihedral loci of the
pseudo-deformation space, and the comparison between the two kinds of rings.
Section~4 recalls the required Hecke algebras, ordinary finiteness, and
patching at nice primes, before proving the componentwise pro-modularity
theorem.  We include a summary of the patching input in
Remark~\ref{dim of nice}; the new part of the argument is the construction and
propagation of nice primes in Subsections~4.4 and~4.5.

\subsection{Notations and conventions}
Throughout the paper, $F$ is a totally real field and $\mathbb{A}_F$ is the adele ring. For a finite place $v$ of $F$, we write $F_v$ for the completion of $F$ at $v$, $\mathcal{O}_{F_v}$ for its ring of integers, $\pi_v$ for a uniformizer and $k(v)$ for the residue field. We denote the absolute Galois group and the inertia group of $F_v$ by $G_{F_v}$ and $I_v$, respectively. For a rational prime $l$, we denote $\mathcal{O}_F \otimes_\mathbb{Z} \mathbb{Z}_l$ by $\mathcal{O}_{F,l}$. We write $\operatorname{Nm}(v)$ for the norm of $v$, and write $\operatorname{Frob}_v$ for the geometric Frobenius element. 

We fix a finite field $\mathbb{F}$ of characteristic $p\ge3$. Let $\mathcal{O}$ be a ring of integers of some local field $E$ (a finite extension of $\mathbb{Q}_p$) with a uniformizer $\pi$ and residue field $\mathbb{F}$. Let $\mathscr{C}_{\mathcal{O}}^{f}$ be the category of Artinian local $\mathcal{O}$-algebras with residue field $\mathbb{F}$, and $\mathscr{C}_{\mathcal{O}}$ be the category of topological local $\mathcal{O}$-algebras which are isomorphic to inverse limits of objects of $\mathscr{C}_{\mathcal{O}}^{f}$. For a universal (pseudo-)deformation ring, we mean a universal object in $\mathscr{C}_{\mathcal{O}}$ which pro-represents the functor from $\mathscr{C}_{\mathcal{O}}^{f}$ to the category of sets sending $R$ to the set of (pseudo-)deformations over $R$.

For a CNL ring $R$, we mean that $R$ is a complete Noetherian local ring with maximal ideal $\mathfrak{m}_R$. If $\mathfrak{p}$ is a prime ideal of $R$, we will denote its residue field by $k(\mathfrak{p})$. We write $R_{\mathfrak{p}}$ for the localization at $\mathfrak{p}$, and $\widehat{R_{\mathfrak{p}}}$ for the $\mathfrak{p}$-adic completion. We write $R^{\operatorname{red}}$ for the maximal reduced quotient of $R$. By $M_i(R)$, we mean the ring of $i \times i$ matrices over $R$.

We write $\epsilon, \omega_p$ for the $p$-adic, $\operatorname{mod}~p$ cyclotomic character, respectively.

\begin{ack}
	The author would like to express appreciation to his supervisor Professor
	Takeshi Saito for his careful reading, pointing out some mathematical errors
	and many helpful suggestions and comments. During the revision of this
	article, the author used OpenAI's Codex to assist with the organization of
	the exposition, English-language editing, and the checking of citations and
	internal cross-references. 
\end{ack}

\section{Pseudo-representations}

Pseudo-representations are the main tools to study the modularity of Galois representations in the residually reducible case. We recommend \cite{bellaiche2009families} to readers for a general theory of pseudo-representations.

\subsection{Two-dimensional pseudo-representations}
In this section, we give some results on 2-dimensional pseudo-representations. Our main references are \cite{skinner1999residually} and \cite[Section 2.1]{pan2022fontaine}.

\begin{defn}\label{pseudo}
	For a profinite group $G$ and a topological commutative ring $R$ in which $2$ is invertible, a 2-dimensional \textit{pseudo-representation} of $G$ over $R$ is a continuous function $T:G\to R$ such that
	
	1) $T(1)=2$;
	
	2) $T(\sigma\tau)=T(\tau\sigma)$ for all $\sigma,\tau\in G$;
	
	3) $T(\gamma\delta\eta)+T(\gamma\eta\delta)-T(\gamma\eta)T(\delta)-T(\eta\delta)T(\gamma)-T(\delta\gamma)T(\eta)+T(\gamma)T(\delta)T(\eta)=0$ for all $\delta,\gamma,\eta\in G$.
	
	The \textit{determinant} $\operatorname{det}(T)$ is the character defined by
	$$
		\operatorname{det}(T):G\to R^\times,\qquad
		\operatorname{det}(T)(\sigma)=\frac{1}{2}\bigl(T(\sigma)^2-T(\sigma^2)\bigr).
	$$
	
	Whenever an element $c\in G$ of order $2$ with $T(c)=0$ is fixed, we define
	\begin{enumerate}
		\item $a(\sigma)=\frac{1}{2}(T(c\sigma)+T(\sigma)),$
		\item $d(\sigma)=\frac{1}{2}(-T(c\sigma)+T(\sigma)),$
		\item $y(\sigma,\tau)=a(\sigma\tau)-a(\sigma)a(\tau),  \sigma,\tau\in G.$
	\end{enumerate}
\end{defn}

In our applications, $c\in G_{F, \Sigma}$ will be a complex conjugation, where $F$ is some totally real field and $\Sigma$ is a finite set of finite places of $F$. The pseudo-representation $T$ will be the trace of some two-dimensional, continuous, totally odd residual Galois representation of $G_{F, \Sigma}$. Thus, $T(c)=0$.

If $\rho:G\to\operatorname{GL}_2(R)$ is a continuous representation such that $\rho(c)=\begin{pmatrix}1&0\\0&-1\end{pmatrix}$ and $\rho(\sigma)=\begin{pmatrix}a_\sigma&b_\sigma\\c_\sigma&d_\sigma\end{pmatrix}$, then $T=\operatorname{tr}(\rho)$ satisfies $a(\sigma)=a_\sigma$, $d(\sigma)=d_\sigma$, and $y(\sigma,\tau)=b_\sigma c_\tau$.

Pan \cite[Section 2.1.4]{pan2022fontaine} verifies that the continuous functions $\{a,d,y\}$ satisfy the following equations.

(1) $y(\sigma,\tau)=d(\tau\sigma)-d(\tau)d(\sigma)$.

(2) $y(\sigma\tau,\delta)=a(\sigma)y(\tau,\delta)+y(\sigma,\delta)d(\tau)$.

(3) $y(\sigma,\tau\delta)=a(\delta)y(\sigma,\tau)+y(\sigma,\delta)d(\tau)$.

(4) $y(\alpha,\beta)y(\sigma,\tau)=y(\alpha,\tau)y(\sigma,\beta)$.

From the last equation, we obtain the following two facts.

\begin{prop}\label{u}
	Let $R$ be an integral domain, let $T:G\to R$ be a pseudo-representation, and let $\alpha,\beta\in G$.
	
	1) If $y(\alpha,\beta)=0$, then either $y(\alpha,\tau)=0$ for every $\tau\in G$, or $y(\sigma,\beta)=0$ for every $\sigma\in G$.
	
	2) Assume that $J_\beta=\{\tau\in G:y(\beta,\tau)\neq0\}$ (resp. $J'_\beta=\{\tau\in G:y(\tau,\beta)\neq0\}$) is non-empty. Then $\frac{y(\alpha,\tau)}{y(\beta,\tau)}$ (resp. $\frac{y(\tau,\alpha)}{y(\tau,\beta)}$) is independent of the choice of $\tau\in J_\beta$ (resp. $\tau\in J'_\beta$). We denote it by $u(\alpha,\beta)$ (resp. $u'(\alpha,\beta)$).
\end{prop}

\begin{proof}
	If $y(\alpha,\beta)=0$ and the first alternative in 1) does not hold, choose $\tau_0\in G$ such that $y(\alpha,\tau_0)\neq0$. By (4), $0=y(\alpha,\beta)y(\sigma,\tau_0)=y(\alpha,\tau_0)y(\sigma,\beta)$ for every $\sigma\in G$, and the second alternative follows because $R$ is a domain.
	
	For $\tau,\tau'\in J_\beta$, (4) gives $y(\alpha,\tau)y(\beta,\tau')=y(\alpha,\tau')y(\beta,\tau)$. This proves that $u(\alpha,\beta)$ is well-defined. The proof for $u'(\alpha,\beta)$ is the same after interchanging the two variables of $y$.
\end{proof}

\begin{defn}\label{construction}\cite[Section 2.1.5]{pan2022fontaine}
	Assume that $R$ is either a field or a discrete valuation ring. For a pseudo-representation $T:G\to R$, a representation $\rho_T$ associated to $T$ is defined as follows.
	
	1) If $y(\sigma,\tau)=0$ for all $\sigma,\tau\in G$, define $\rho_T(\sigma)=\begin{pmatrix}a(\sigma)&0\\0&d(\sigma)\end{pmatrix}$. We call this case reducible.
	
	2) Otherwise, choose $\sigma_0,\tau_0\in G$ such that $y(\sigma_0,\tau_0)\neq0$ and $\frac{y(\sigma,\tau)}{y(\sigma_0,\tau_0)}\in R$ for all $\sigma,\tau\in G$. If $R$ is a field, any pair with $y(\sigma_0,\tau_0)\neq0$ will do; if $R$ is a discrete valuation ring, we choose a pair for which the valuation of $y(\sigma_0,\tau_0)$ is minimal. Define
	$$
		\rho_T(\sigma)=
		\begin{pmatrix}
			a(\sigma)&\dfrac{y(\sigma,\tau_0)}{y(\sigma_0,\tau_0)}\\[4pt]
			y(\sigma_0,\sigma)&d(\sigma)
		\end{pmatrix}.
	$$
	We call this case irreducible.
\end{defn}

\begin{rem}
	The identities above show that $\rho_T$ is a continuous representation satisfying $\operatorname{tr}(\rho_T)=T$ and $\operatorname{det}(\rho_T)=\operatorname{det}(T)$. In the irreducible case, its generic fibre is absolutely irreducible; if $R$ is a field, it is unique up to conjugation. See \cite[Section 2.1.5]{pan2022fontaine}.
\end{rem}

\begin{prop}\label{innovation}
	Let $R$ be a CNL domain with maximal ideal $\mathfrak{m}_R$ and residue field $\mathbb{F}$. Let $G$ be a profinite group and let $T:G\to R$ be a pseudo-representation such that $T~\operatorname{mod}~\mathfrak{m}_R$ is reducible. For $\alpha\in G$, write $I_\alpha$ for the ideal of $R$ generated by $\{y(\alpha,\beta):\beta\in G\}$. Assume that every minimal prime over $I_\alpha$ (in the terminology of \cite[page 52]{MR242802}) has height at least $2$. Then we have the following results.
	
	1) For every $\sigma\in G$, $u(\sigma,\alpha)$ is well-defined and integral over $R$.
	
	2) If $\alpha'\in G$ satisfies the analogous height condition for $I_{\alpha'}$, then $u(\alpha,\alpha')$ is a unit in the normalization of $R$.
\end{prop}

\begin{proof}
	Let $B$ be the normalization of $R$ in $\operatorname{Frac}(R)$, and let $i:R\hookrightarrow B$ be the natural inclusion. Since a complete Noetherian local ring is Nagata, $B$ is finite over $R$; see \cite[Tags 032W and 035S]{stacks-project}. Moreover, $R$ is Henselian, so the finite $R$-algebra $B$, which is a domain, is local. Equivalently, $R$ is unibranch (see \cite[Tag 04GM and Tag 0BQ0]{stacks-project}).
	
	We first recall a useful result in commutative algebra.
	
	\begin{lemma}\label{commutative algebra}
		If $\mathfrak{p}$ is a prime ideal of $B$ of height $d$, then $i^{-1}(\mathfrak{p})$ also has height $d$ in $R$.
	\end{lemma}
	
	\begin{proof}
		This is \cite[Proposition 9.2]{eisenbud2013commutative}.
	\end{proof}
	
	We return to the proposition. Since the residual pseudo-representation is reducible, $y(\alpha,\beta)\in\mathfrak{m}_R$ for all $\alpha,\beta\in G$. The height assumption implies $I_\alpha\neq0$, so $u(\sigma,\alpha)$ is well-defined by Proposition~\ref{u}.
	
	1) Let $\mathfrak{p}_B$ be a height-one prime of $B$. By Lemma~\ref{commutative algebra}, $i^{-1}(\mathfrak{p}_B)$ has height one and therefore cannot contain $I_\alpha$. Hence there is some $\beta\in G$ such that $y(\alpha,\beta)\notin\mathfrak{p}_B$, and
	$$
		u(\sigma,\alpha)=\frac{y(\sigma,\beta)}{y(\alpha,\beta)}\in B_{\mathfrak{p}_B}.
	$$
	Since $B$ is a Noetherian normal domain, $B=\bigcap_{\operatorname{ht}(\mathfrak{p}_B)=1}B_{\mathfrak{p}_B}$ inside $\operatorname{Frac}(R)$. Thus $u(\sigma,\alpha)\in B$, which proves 1).
	
	2) By 1), both $u(\alpha',\alpha)$ and $u(\alpha,\alpha')$ belong to $B$. Since $u(\alpha',\alpha)u(\alpha,\alpha')=1$, we have $u(\alpha,\alpha')\in B^\times$.
\end{proof}

\begin{cor}\label{innovation'}
	Let $R$ be a CNL domain with maximal ideal $\mathfrak{m}_R$ and residue field $\mathbb{F}$ of characteristic $p>2$. Let $G$ be a profinite group and let $T:G\to R$ be a pseudo-representation such that $T~\operatorname{mod}~\mathfrak{m}_R$ is reducible. Let $S$ be a finite subset (not necessarily a subgroup) of $G$. Then there exist a partition $S=S_1\amalg S_2$, an integer $n>1$ satisfying $(n,p)=1$, and a CNL domain $R'$ satisfying the following conditions.
	
	a) $R'$ is a quotient of $R$.
	
	b) For every $\theta\in S_1$, we have $y(\theta,\alpha)=0$ for every $\alpha\in G$ in $R'$. For every $\theta',\theta''\in S_2$, we have $y(\theta',\alpha)^n=y(\theta'',\alpha)^n$ for every $\alpha\in G$ in $R'$.
	
	c) For $\alpha\in G$ and $\theta\in S_2$, either $y(\theta,\beta)=0$ for every $\beta\in G$ in $R'$, or $u(\alpha,\theta)$ is well-defined and integral over $R'$.
	
	d) We have $\operatorname{dim}R'\geq\operatorname{dim}R-|S|$. If $S_2$ is non-empty, then $\operatorname{dim}R'\geq\operatorname{dim}R-|S|+1$.
\end{cor}

\begin{proof}
	We argue by induction on $|S|$. If $S=\varnothing$, take $S_1=S_2=\varnothing$, $R'=R$, and $n=2$.
	
	Assume that $S\neq\varnothing$, and write $I_\theta=(y(\theta,\beta):\beta\in G)$ for $\theta\in S$. We distinguish two cases.
	
	First, suppose that every minimal prime over every $I_\theta$ has height at least $2$. Let $B$ be the normalization of $R$. As in the proof of Proposition~\ref{innovation}, $B$ is finite and local, so its residue field $k_B$ is finite. Take $S_1=\varnothing$, $S_2=S$, and $n=|k_B^\times|=|k_B|-1$. Then $n>1$ and $(n,p)=1$.
	
	Fix $\theta_1\in S$. By Proposition~\ref{innovation}, $u(\theta_1,\theta)\in B^\times$ for every $\theta\in S$, and hence $u(\theta_1,\theta)^n-1\in\mathfrak{m}_B$. Let $\mathfrak{b}$ be the ideal of $B$ generated by $\{u(\theta_1,\theta)^n-1:\theta\in S\}$. Since $u(\theta_1,\theta_1)=1$, the ideal $\mathfrak{b}$ is generated by at most $|S|-1$ non-zero elements. By Krull's height theorem, there is a minimal prime $\mathfrak{p}$ over $\mathfrak{b}$ with $\operatorname{ht}(\mathfrak{p})\leq|S|-1$. Lemma~\ref{commutative algebra} gives $\operatorname{ht}(i^{-1}(\mathfrak{p}))\leq|S|-1$. Take $R'=R/i^{-1}(\mathfrak{p})$.
	
	For $\theta\in S$ and $\alpha\in G$, we have $y(\theta_1,\alpha)=u(\theta_1,\theta)y(\theta,\alpha)$ in $B$. Taking $n$-th powers and reducing modulo $\mathfrak{p}$ gives $y(\theta_1,\alpha)^n=y(\theta,\alpha)^n$ in $R'$, which proves b). If the $\theta$-row does not vanish in $R'$, choose $\beta$ such that $y(\theta,\beta)\neq0$ in $R'$. The image of the integral element $u(\alpha,\theta)\in B$ equals $\frac{y(\alpha,\beta)}{y(\theta,\beta)}\in\operatorname{Frac}(R')$, and is therefore integral over $R'$. This proves c). Since a complete local domain is catenary, $\operatorname{dim}R'=\operatorname{dim}R-\operatorname{ht}(i^{-1}(\mathfrak{p}))\geq\operatorname{dim}R-|S|+1$, proving d).
	
	Second, suppose that there are $\theta\in S$ and a minimal prime $\mathfrak{p}_\theta$ over $I_\theta$ with $\operatorname{ht}(\mathfrak{p}_\theta)\leq1$. Set $R_1=R/\mathfrak{p}_\theta$. Then $R_1$ is a CNL domain, $y(\theta,\alpha)=0$ in $R_1$ for every $\alpha\in G$, and, by the same dimension formula, $\operatorname{dim}R_1\geq\operatorname{dim}R-1$. Apply the induction hypothesis to $R_1$ and $S\setminus\{\theta\}$. If $S\setminus\{\theta\}=S_{1,\mathrm{ind}}\amalg S_{2,\mathrm{ind}}$ is the resulting partition, take $S_1=S_{1,\mathrm{ind}}\cup\{\theta\}$ and $S_2=S_{2,\mathrm{ind}}$. The conditions a)--c) follow immediately, and the required dimension bounds follow from $\operatorname{dim}R_1\geq\operatorname{dim}R-1$.
\end{proof}

We shall also use the transpose form of Corollary~\ref{innovation'}, obtained by interchanging the two variables of $y$ and replacing $u$ by $u'$. The fourth identity above is invariant under this transposition, so the same proof and the same dimension bounds apply. We do not number this form separately in order to keep the numbering used in subsequent papers.

\begin{rem}
	This result is also used in \cite{ZHANG2024}. 
\end{rem}

\section{Galois deformation theory}
In this section, we study some properties of 2-dimensional universal deformation rings and pseudo-deformation rings.

\subsection{Universal deformation rings}	
In this subsection, we study some properties of universal deformation rings related to $\rho_x$.

Let $F$ be a totally real field. Let $p\ge3$ be a rational prime and $\mathbb{F}$ a finite field of characteristic $p$. Let $\Sigma$ be a finite set of finite places of $F$ containing the set $\Sigma_p$ of all places $v\mid p$. Recall that $\mathcal{O}$ is the ring of integers of a finite extension $E/\mathbb{Q}_p$, with uniformizer $\pi$ and residue field $\mathbb{F}$. Let $\chi:\operatorname{Gal}(\bar{F}/F)\to\mathcal{O}^\times$ be a continuous totally odd character unramified outside $\Sigma$, and let $\bar{\chi}=\chi~\operatorname{mod}~\pi$. Let $F_{\Sigma}$ be the maximal extension of $F$ unramified outside $\Sigma$ and all infinite places, and put $\bar{\rho}_0=1\oplus\bar{\chi}$.

For an element $0\neq x\in H^1(F_{\Sigma}/F,\mathbb{F}(\bar{\chi}^{-1}))$, let $\rho_x$ be the non-split extension determined by $x$, uniquely up to equivalence, with
$$
\rho_x:\operatorname{Gal}(F_{\Sigma}/F)\to\operatorname{GL}_2(\mathbb{F}),\qquad
\rho_x=
\begin{pmatrix}
	1 &  *\\
	~& \bar{\chi}
\end{pmatrix}.
$$
After a change of basis, we may assume that
$$
\rho_x(c)=
\begin{pmatrix}
	1 & ~\\
	~& -1
\end{pmatrix}
$$
for a complex conjugation $c$, where $*$ represents the cohomology class $x$.

Since $\bar{\chi}$ is totally odd and $p>2$, we have $\bar{\chi}\neq1$. As $\rho_x$ is non-split, $\operatorname{End}_{\mathbb{F}[\operatorname{Gal}(F_{\Sigma}/F)]}(\rho_x)=\mathbb{F}$, so the deformation functor of $\rho_x$ is representable.

Let $R_x$ be the universal deformation ring of $\rho_x$ in the sense of Mazur \cite{mazur1989deforming}, and let $\rho_x^\textnormal{univ}:\operatorname{Gal}(F_{\Sigma}/F)\to\operatorname{GL}_2(R_x)$ be the universal deformation. Following \cite[Proposition 2.4]{skinner1999residually}, we consider the deformation problem of auxiliary type.

\begin{defn}\label{auxiliary}
	A deformation $\rho:\operatorname{Gal}(\bar{F}/F)\to\operatorname{GL}_2(A)$ of $\rho_x$ is of auxiliary type if
	
	1) $A$ is an $\mathcal{O}$-algebra,
	
	2) $\operatorname{det}\rho=\chi$,
	
	3) $\rho$ factors through $\operatorname{Gal}(F_{\Sigma}/F)$.
	
	We denote the universal deformation of auxiliary type by
	$\rho_x^\textnormal{aux}:\operatorname{Gal}(F_{\Sigma}/F)\to\operatorname{GL}_2(R_x^\textnormal{aux})$.
\end{defn}

\begin{prop}\label{dimension}
	The ring $R_x^\textnormal{aux}$ is isomorphic to $\mathcal{O}[[x_1,\cdots,x_g]]/(f_1,\cdots,f_r)$, where $g=\operatorname{dim}_{\mathbb{F}}H^1(F_{\Sigma}/F,\operatorname{ad}^0\rho_x)$ and $r\leq\operatorname{dim}_{\mathbb{F}}H^2(F_{\Sigma}/F,\operatorname{ad}^0\rho_x)$.
\end{prop}

\begin{proof}
	This follows from the tangent-obstruction argument of \cite[Proposition 2]{mazur1989deforming}.
\end{proof}
 
\begin{prop}\label{raux}
	The dimension of every irreducible component of $R_x^\textnormal{aux}$ is at least $2[F:\mathbb{Q}]+1$.
\end{prop}
 
\begin{proof}
	By Proposition~\ref{dimension} and Krull's height theorem, every irreducible component has dimension at least $1+g-r$. Since $\rho_x$ is non-split and $\bar{\chi}\neq1$, we have $H^0(F_{\Sigma}/F,\operatorname{ad}^0\rho_x)=0$. Moreover, $\operatorname{dim}_{\mathbb{F}}H^0(G_{F_v},\operatorname{ad}^0\rho_x)=1$ for every $v\mid\infty$. Hence the global Euler characteristic formula gives
	$$
		\operatorname{dim}_{\mathbb{F}}H^1(F_{\Sigma}/F,\operatorname{ad}^0\rho_x)
		-\operatorname{dim}_{\mathbb{F}}H^2(F_{\Sigma}/F,\operatorname{ad}^0\rho_x)
		=2[F:\mathbb{Q}].
	$$
	Since $r\leq\operatorname{dim}_{\mathbb{F}}H^2(F_{\Sigma}/F,\operatorname{ad}^0\rho_x)$, we obtain $g-r\geq2[F:\mathbb{Q}]$. This is also the auxiliary deformation problem appearing in the proof of \cite[Proposition 2.4]{skinner1999residually}.
\end{proof}
 
Let $\rho_x^\textnormal{aux,red}:\operatorname{Gal}(F_{\Sigma}/F)\to\operatorname{GL}_2(R_x^\textnormal{aux,red})$ be the universal reducible deformation of $\rho_x$ of auxiliary type. Then $R_x^\textnormal{aux,red}$ is a quotient of $R_x^\textnormal{aux}$.
 
\begin{prop}\label{red}
	Let $\delta_F$ be the Leopoldt defect (see \cite[Section 10.3.7]{neukirch2013cohomology}). Then $\operatorname{dim}(R_x^\textnormal{aux,red})\leq1+\delta_F+\operatorname{dim}_{\mathbb{F}}H^1(F_{\Sigma}/F,\mathbb{F}(\bar{\chi}^{-1}))$.
 	
\end{prop}
 
\begin{proof}
	The proof is analogous to \cite[Section 2.2]{skinner1999residually}; we include the construction in the fixed-determinant setting.
	
	Let $L(\Sigma)$ be the maximal abelian pro-$p$-extension of $F$ unramified outside $\Sigma$, and write, non-canonically, $G=\operatorname{Gal}(L(\Sigma)(\bar{\chi})/F)\simeq\Delta\times\Gamma\times\mathbb{Z}_p^{\delta_F+1}$, where $\Delta=\operatorname{Gal}(F(\bar{\chi})/F)$, $\Gamma$ is a finite $p$-group, and $L(\Sigma)(\bar{\chi})=L(\Sigma)\cdot F(\bar{\chi})$. Let $M$ be the maximal abelian pro-$p$-extension of $L(\Sigma)(\bar{\chi})$ unramified outside $\Sigma$ such that $\Delta$ acts on
	$H=\operatorname{Gal}(M/L(\Sigma)(\bar{\chi}))$ through the unique $\mathbb{Z}_p$-representation associated to $\bar{\chi}^{-1}$, as in \cite[Section 2.2]{skinner1999residually}. Then every reducible deformation of auxiliary type factors through $\operatorname{Gal}(M/F)$.
	
	Put $\Lambda=\mathbb{Z}_p[[G]]$, and let $\mathfrak{m}_{\bar{\chi}^{-1}}$ be the maximal ideal of $\Lambda$ determined by $\bar{\chi}^{-1}$ on $\Delta$ and by the augmentation on the pro-$p$ part. Its residue field naturally embeds in $\mathbb{F}$. If $m=\operatorname{dim}_{\mathbb{F}}H^1(F_{\Sigma}/F,\mathbb{F}(\bar{\chi}^{-1}))$, the argument in \cite[Section 2.2]{skinner1999residually} identifies $m$ with the minimal number of generators of $H$ as a $\Lambda$-module. Thus, by Nakayama's lemma, we may fix a presentation
	$\mathfrak{a}\to\bigoplus_{i=1}^m\Lambda e_i\twoheadrightarrow H$
	such that $e_m$ maps to an element $h_m$ with
	$\rho_x(h_m)=\begin{pmatrix}1&u\\0&1\end{pmatrix}$ for some $u\neq0$, while $e_i$ maps to an element $h_i$ with $\rho_x(h_i)=1$ for $i\neq m$. Choose $u_0\in\mathcal{O}^\times$ reducing to $u$.
	
	Put $A_1=\mathcal{O}[[\Gamma]][[T_1,\cdots,T_{\delta_F+1}]]$, and let $\psi:G\to A_1^\times$ be the universal character whose restriction to $\Delta$ is the Teichm\"uller lift of $\bar{\chi}$. The character $\chi$ factors through $G$, since its quotient by the Teichm\"uller lift of $\bar{\chi}$ has pro-$p$ image and is abelian and unramified outside $\Sigma$. Let $\phi:\Lambda\to A_1$ be the continuous homomorphism induced by the character $g\mapsto\chi(g)\psi(g)^{-2}$. Notice that its reduction is $\bar{\chi}^{-1}$, as required by the action of $\Delta$ on $H$.
	
	Let $J$ be the ideal of $A_1[[x_1,\cdots,x_{m-1}]]$ generated by
	$$
		\left\{\phi(a_1)x_1+\cdots+\phi(a_{m-1})x_{m-1}
		+\phi(a_m)u_0:\sum_{i=1}^m a_i e_i\in\mathfrak{a}\right\},
	$$
	and put $B=A_1[[x_1,\cdots,x_{m-1}]]/J$. The relations defining $J$ give a $\Lambda$-linear homomorphism $\tau:H\to B$, where $\Lambda$ acts on $B$ through $\phi$, such that $\tau(h_i)=x_i$ for $1\leq i\leq m-1$ and $\tau(h_m)=u_0$.
	
	Using $\operatorname{Gal}(M/F)=H\rtimes G$, define
	$$
		\rho(g)=
		\begin{pmatrix}
			\chi(g)\psi(g)^{-1}&0\\
			0&\psi(g)
		\end{pmatrix}\quad(g\in G),\qquad
		\rho(h)=
		\begin{pmatrix}
			1&\tau(h)\\
			0&1
		\end{pmatrix}\quad(h\in H).
	$$
	The identity $\phi(g)=\chi(g)\psi(g)^{-2}$ gives the required compatibility with the semidirect product. Conversely, every reducible deformation of auxiliary type can be put in this form by strict conjugation: its lower diagonal character determines $\psi$, and its upper-right entry determines a $\Lambda$-linear map from $H$. Since the image of $h_m$ is a unit lifting $u$, a further strict diagonal conjugation normalizes it to $u_0$. Thus, as in \cite[Section 2.2]{skinner1999residually}, this is the universal reducible deformation of auxiliary type, so $B\simeq R_x^\textnormal{aux,red}$. Therefore $\operatorname{dim}(R_x^\textnormal{aux,red})\leq\operatorname{dim}A_1+(m-1)=1+\delta_F+m$, which is the desired inequality.
 	 
 	 
 	 
\end{proof}

\begin{prop}\label{char}
	Assume that $p\mid\operatorname{Nm}(v)-1$ for every $v\in\Sigma\setminus\Sigma_p$, and that $\chi$ is unramified outside $\Sigma_p$. Let $R$ be a CNL $\mathcal{O}$-algebra domain with maximal ideal $\mathfrak{m}_R$, fraction field $K$, and residue field $\mathbb{F}$. Suppose that
	$\rho:\operatorname{Gal}(F_{\Sigma}/F)\to\operatorname{GL}_2(R)$
	has determinant $\chi$ and satisfies $\operatorname{tr}(\rho)|_{I_v}\equiv2~\operatorname{mod}~\pi$ for every $v\in\Sigma\setminus\Sigma_p$. Then, for every such $v$, there exist a finite extension $K'/K$ and a finite character $\xi_v:I_v\to K'^\times$ of $p$-power order such that $\operatorname{tr}(\rho)|_{I_v}=\xi_v+\xi_v^{-1}$.
\end{prop}
 
\begin{proof}
	This is analogous to \cite[Corollary 5.7.2]{pan2022fontaine}; we include the argument. Fix $v\in\Sigma\setminus\Sigma_p$. Since $\chi$ is unramified at $v$, the characteristic polynomial of $\rho(\sigma)$ is congruent to $(X-1)^2$ modulo $\mathfrak{m}_R$ for every $\sigma\in I_v$. The residual image of $I_v$ therefore consists of unipotent elements and is a finite $p$-group, while the kernel of reduction is pro-$p$. Thus, the image of $I_v$ under $\rho$ is pro-$p$. Since $v\nmid p$, the description of tame inertia in \cite[Proposition 7.5.2]{neukirch2013cohomology} shows that $\rho|_{I_v}$ factors through the $\mathbb{Z}_p$-quotient of $I_v$. Let $\sigma_v$ be a topological generator of this quotient.
	
	After replacing $K$ by a finite extension $K'$, let $\alpha,\alpha^{-1}\in K'^\times$ be the eigenvalues of $\rho(\sigma_v)$. Let $R'$ be the localization at a maximal ideal above $\mathfrak{m}_R$ of the normalization of $R$ in $K'$. Since $\alpha$ is integral over $R$, we have $\alpha\in R'$ and $\alpha\equiv1~\operatorname{mod}~\mathfrak{m}_{R'}$. By \cite[Theorem 7.5.3]{neukirch2013cohomology}, there exists $\tau\in G_{F_v}\setminus I_v$ such that
	$\tau\sigma_v\tau^{-1}=\sigma_v^{\operatorname{Nm}(v)}$. Comparing eigenvalues on the two sides gives
	$\alpha^{\operatorname{Nm}(v)^2-1}=1$. Since prime-to-$p$ roots of unity inject into the residue field, $\alpha$ has $p$-power order.
	
	Define $\xi_v$ on the $\mathbb{Z}_p$-quotient of $I_v$ by $\xi_v(\sigma_v)=\alpha$. It is a finite character of $p$-power order, and for every $\sigma\in I_v$ the trace of $\rho(\sigma)$ is the sum of its two eigenvalues. Hence $\operatorname{tr}(\rho)|_{I_v}=\xi_v+\xi_v^{-1}$.
\end{proof}
 
 
 
 
 
 

 \subsection{Universal pseudo-deformation ring} \label{mis} In this subsection, we collect some results which will be useful for our pro-modularity arguments.
 
 Keep the notations and assumptions of the previous subsection. Write $\Sigma_p\subset\Sigma$ for the set of places of $F$ lying above $p$. We further assume that $p\mid\operatorname{Nm}(v)-1$ for every $v\in\Sigma\setminus\Sigma_p$, and that $\chi:\operatorname{Gal}(F_{\Sigma}/F)\to\mathcal{O}^{\times}$ satisfies

 1) $\chi$ is unramified outside $\Sigma_p$,
 
 2) $\chi(\operatorname{Frob}_v)\equiv1~\operatorname{mod}~\pi$ for $v\in\Sigma\setminus\Sigma_p$.
 
Let $R^\textnormal{ps}$ be the universal pseudo-deformation ring of $\operatorname{tr}(\bar{\rho}_0)=1+\bar{\chi}$, and let $T^\textnormal{univ}$ be the universal pseudo-deformation. The existence of $(R^\textnormal{ps},T^\textnormal{univ})$ is proved in \cite[Section 2.4]{skinner1999residually}. Let $(R_\textnormal{aux}^\textnormal{ps},T_\textnormal{aux})$ be the universal pseudo-deformation of auxiliary type, i.e., with determinant $\chi$.

For each $v\in\Sigma\setminus\Sigma_p$, let $\xi_v:I_v\to\mathcal{O}^\times$ be a finite character of $p$-power order. Whenever a finite extension $E'/E$ of the coefficient field is necessary, we replace the deformation rings by the corresponding completed base changes to the ring of integers $\mathcal{O}'$ and, when necessary, by the relevant complete local factors. We then keep the same notation; in particular, $\pi$ always denotes a uniformizer of the current coefficient ring. We denote by $(R^{\textnormal{ps},\{\xi_v\}},T^{\textnormal{ps},\{\xi_v\}})$ the universal pseudo-deformation of auxiliary type satisfying $T^{\textnormal{ps},\{\xi_v\}}|_{I_v}=\xi_v+\xi_v^{-1}$ for every $v\in\Sigma\setminus\Sigma_p$. There are natural surjections $R^\textnormal{ps}\twoheadrightarrow R_\textnormal{aux}^\textnormal{ps}\twoheadrightarrow R^{\textnormal{ps},\{\xi_v\}}$. If $C=V(\mathfrak{P})$ is an irreducible component of $\operatorname{Spec}R_\textnormal{aux}^\textnormal{ps}$, we say that $C$ \textit{belongs to} $\{\xi_v\}$ if the natural surjection $R_\textnormal{aux}^\textnormal{ps}\to R_\textnormal{aux}^\textnormal{ps}/\mathfrak{P}$ factors through $R^{\textnormal{ps},\{\xi_v\}}$.

The following facts will be used later.

\begin{prop}\label{RX1}
	After a finite extension of $E$ if necessary, the following assertions hold.
		
	1) Every irreducible component of $\operatorname{Spec}R_\textnormal{aux}^\textnormal{ps}$ belongs to some system $\{\xi_v\}$.
	
	2) For any two systems $\{\xi_v\}$ and $\{\xi'_v\}$ as above, we have $R^{\textnormal{ps},\{\xi_v\}}/\pi=R^{\textnormal{ps},\{\xi'_v\}}/\pi$ as quotients of $R_\textnormal{aux}^\textnormal{ps}/\pi$.
\end{prop}

\begin{proof}
	1) Let $C=V(\mathfrak{P})$ be an irreducible component and put $A=R_\textnormal{aux}^\textnormal{ps}/\mathfrak{P}$. After passing to $\operatorname{Frac}(A)$, the pseudo-representation $T_\textnormal{aux}~\operatorname{mod}~\mathfrak{P}$ has an associated semisimple representation in the sense of Definition \ref{construction}. By \cite[Corollary 5.7.2]{pan2022fontaine}, after a finite extension there are finite characters $\theta_{v,1},\theta_{v,2}$ such that $(T_\textnormal{aux}~\operatorname{mod}~\mathfrak{P})|_{I_v}=\theta_{v,1}+\theta_{v,2}$. Since the determinant is unramified at $v$, we have $\theta_{v,2}=\theta_{v,1}^{-1}$. Moreover, both characters reduce to the trivial character, so their prime-to-$p$ parts are trivial. Hence $\xi_v:=\theta_{v,1}$ has $p$-power order. After the coefficient extension and base change stipulated above, the relation holds on the corresponding component, and the map to its coordinate domain factors through $R^{\textnormal{ps},\{\xi_v\}}$. Since there are only finitely many irreducible components and finitely many places in $\Sigma\setminus\Sigma_p$, one common finite extension of the coefficient field suffices.
	
	2) We first pass to a common coefficient field containing the values of both systems and let $\pi$ denote a uniformizer of its ring of integers. Every character of $p$-power order is trivial modulo $\pi$. Thus, for either system, the local condition modulo $\pi$ is $T|_{I_v}=2$ for every $v\in\Sigma\setminus\Sigma_p$, and the two defining quotients coincide.
\end{proof}

Let $(R^{\textnormal{ps,red}},T^{\textnormal{red}})$ be the universal reducible pseudo-deformation of auxiliary type; equivalently, $y(\sigma,\tau)=0$ for all $\sigma,\tau\in\operatorname{Gal}(F_{\Sigma}/F)$.

\begin{prop}\label{red-pseudo} Assume $F$ is abelian over $\mathbb{Q}$.
	
	1) We have $\operatorname{dim}(R^{\textnormal{ps,red}})\leq2$ and $\operatorname{dim}(R^{\textnormal{ps,red}}/\pi)\leq1$.
	
	2) Assume that $\bar{\chi}$ is quadratic and extends to $G_{\mathbb{Q}}$, and put $L=F(\bar{\chi})$. Let $C^{\operatorname{dih}}$ be the Zariski closure in $\operatorname{Spec}R_\textnormal{aux}^\textnormal{ps}$ of the points $\mathfrak{p}$ for which $\rho(\mathfrak{p})\cong\operatorname{Ind}_{G_L}^{G_F}\theta$ for some character $\theta$. Then $\operatorname{dim}C^{\operatorname{dih}}\leq2+[F:\mathbb{Q}]$ and $\operatorname{dim}(C^{\operatorname{dih}}\cap V(\pi))\leq1+[F:\mathbb{Q}]$.
\end{prop}

\begin{proof}
	1) Let $G_{F,\Sigma}^{\operatorname{ab},p}$ be the maximal abelian pro-$p$ quotient of $\operatorname{Gal}(F_{\Sigma}/F)$. The character $d:\operatorname{Gal}(F_{\Sigma}/F)\to(R^{\textnormal{ps,red}})^\times$ of Definition \ref{pseudo}, after twisting by the Teichm\"uller lift of $\bar{\chi}^{-1}$, induces a natural map $\mathcal{O}[[G_{F,\Sigma}^{\operatorname{ab},p}]]\to R^{\textnormal{ps,red}}$. Since $y=0$, we have $T^{\textnormal{red}}=\chi d^{-1}+d$, so the values of $d$ topologically generate $R^{\textnormal{ps,red}}$ and this map is surjective. The $\mathbb{Z}_p$-rank of $G_{F,\Sigma}^{\operatorname{ab},p}$ is $1+\delta_F$; see \cite[Section 10.3.7]{neukirch2013cohomology}. Since Leopoldt's conjecture holds for abelian totally real fields, we have $\delta_F=0$, and the two dimension bounds follow before and after reduction modulo $\pi$.
	
	2) If $\rho(\mathfrak{p})\cong\operatorname{Ind}_{G_L}^{G_F}\theta$, then $\theta$ lifts $\bar{\chi}|_{G_L}=\mathbf{1}$. Since $\bar{\chi}$ is totally odd and extends to $G_{\mathbb{Q}}$, the field $L$ is an abelian CM field. The character $\theta$ therefore varies in the character-deformation space represented by the Iwasawa algebra $\Lambda_L$ of the maximal abelian pro-$p$ quotient of $G_L$ unramified outside the places above $\Sigma$. By Leopoldt's conjecture for abelian fields, $\Lambda_L$ has dimension $2+[F:\mathbb{Q}]$, and its reduction modulo $\pi$ has dimension $1+[F:\mathbb{Q}]$.
	
	Let $B$ be the coordinate ring of the scheme-theoretic image of the induced pseudo-representation in $\operatorname{Spec}R_\textnormal{aux}^\textnormal{ps}$. After imposing the fixed-determinant condition, for every $h\in G_L$ the universal value $\theta(h)$ satisfies
	$$
		X^2-T(h)X+\chi(h)=0
	$$
	over $B$. Choosing finitely many topological generators of the relevant abelian pro-$p$ quotient, we see that the corresponding quotient of $\Lambda_L$ is finite over $B$. Thus its image is closed and has dimension at most $2+[F:\mathbb{Q}]$; after reduction modulo $\pi$, its dimension is at most $1+[F:\mathbb{Q}]$. Since $C^{\operatorname{dih}}$ is contained in this image, the desired bounds follow.
\end{proof}

\begin{cor}\label{red-one}
	Keep the assumptions of Proposition \ref{red-pseudo}. Let $\mathfrak{s}$ be a prime of $R_\textnormal{aux}^\textnormal{ps}$ containing $p$ and write $R=R_\textnormal{aux}^\textnormal{ps}/\mathfrak{s}$.
	
	1) If $\dim R\geq2$, then there exists a one-dimensional prime $\mathfrak{s}'$ of $R_\textnormal{aux}^\textnormal{ps}$ containing $\mathfrak{s}$ such that the corresponding representation is irreducible.
	
	2) Under the additional assumptions of Proposition \ref{red-pseudo}(2), if $\dim R\geq2+[F:\mathbb{Q}]$, then there exists a one-dimensional prime $\mathfrak{s}'$ of $R_\textnormal{aux}^\textnormal{ps}$ containing $\mathfrak{s}$ such that the corresponding representation is irreducible and not dihedral.
\end{cor}

\begin{proof}
	We first give a useful lemma in commutative algebra.
	
	\begin{lem}\label{Jacobon}
		Let $R$ be a positive-dimensional noetherian local ring with maximal ideal $\mathfrak{m}_R$. Suppose that $V\subsetneqq\operatorname{Spec}R$ is a closed subset. Then there exists a prime $\mathfrak{p}\notin V$ such that $\dim R/\mathfrak{p}=1$.
	\end{lem}
	
	\begin{proof}[Proof of Lemma \ref{Jacobon}]
		By \cite[Lemma 28.6.4]{stacks-project}, the punctured spectrum $U=\operatorname{Spec}R\setminus\{\mathfrak{m}_R\}$ is Jacobson. The open subset $U\setminus V$ is non-empty, so it contains a closed point $\mathfrak{p}$. Being closed in the punctured spectrum is equivalent to $\dim R/\mathfrak{p}=1$.
	\end{proof}
	
	For 1), let $V_\textnormal{red}\subset\operatorname{Spec}R$ be the reducible locus. It is closed, and Proposition \ref{red-pseudo}(1) gives $\dim V_\textnormal{red}\leq\dim(R^{\textnormal{ps,red}}/\pi)\leq1$. Thus, $V_\textnormal{red}$ is a proper closed subset of $\operatorname{Spec}R$, and Lemma \ref{Jacobon} gives a one-dimensional prime outside it.
	
	For 2), let $V$ be the union of the reducible and dihedral loci in $\operatorname{Spec}R$. Since $\operatorname{Spec}R\subset V(\pi)$, Proposition \ref{red-pseudo} gives $\dim V\leq1+[F:\mathbb{Q}]<\dim R$, so $V$ is proper. Another application of Lemma \ref{Jacobon} gives a one-dimensional prime outside both loci. Taking inverse images in $R_\textnormal{aux}^\textnormal{ps}$ proves the assertions.
\end{proof}

Following \cite{skinner1999residually} and \cite{pan2022fontaine}, it is natural to establish the relationship between deformation rings and pseudo-deformation rings.

\begin{prop}\label{pan}
	1) Let $\mathfrak{p}_x$ be a one-dimensional prime of $R_x$. Suppose that $\rho_x^\textnormal{univ}~\operatorname{mod}~\mathfrak{p}_x$ is irreducible. If $\mathfrak{p}=\mathfrak{p}_x\cap R^\textnormal{ps}$, then the natural map $\widehat{(R^\textnormal{ps})_{\mathfrak{p}}}\to\widehat{(R_x)_{\mathfrak{p}_x}}$ is an isomorphism, and $\operatorname{Spec}(R_x)_{\mathfrak{p}_x}\to\operatorname{Spec}(R^\textnormal{ps})_{\mathfrak{p}}$ is surjective.
	
	2) For any $\mathcal{Q}\in\operatorname{Spec}R_x$ such that $\rho_x^\textnormal{univ}~\operatorname{mod}~\mathcal{Q}$ is irreducible, we have $$
	\operatorname{dim} R_x/\mathcal{Q} \le \operatorname{dim} R^\textnormal{ps}/ (\mathcal{Q} \cap R^\textnormal{ps}).$$ Moreover, if $\mathcal{Q}$ is a minimal prime, so is $\mathcal{Q} \cap R^\textnormal{ps}$.
	
	3) The assertions in 1) and 2) also hold for the auxiliary rings $R_x^\textnormal{aux}$ and $R_\textnormal{aux}^\textnormal{ps}$.
	
\end{prop}

\begin{proof}
	The first two assertions are \cite[Corollary 2.3.6]{pan2022fontaine}. The same proof applies after passing to the fixed-determinant quotients, which gives 3).
	Indeed, the comparison morphisms are defined by evaluating traces and are compatible with determinants; imposing the equation $\operatorname{det}=\chi$ therefore preserves the completed isomorphism and the dimension argument.
\end{proof}

\begin{defn}
	Let $R$ be a CNL ring. Following \cite[Definition 5.4.3]{pan2022fontaine}, the connectedness dimension of $R$ is defined by
	$$c(R):=\min_{Z_1,Z_2}\{\dim(Z_1\cap Z_2)\},$$
	where $Z_1$ and $Z_2$ are non-empty unions of irreducible components of $\operatorname{Spec}R$ such that $Z_1\cup Z_2=\operatorname{Spec}R$.
\end{defn}

\begin{prop}\label{connect dim} Suppose that $\mathfrak{p}$ is a one-dimensional prime of $R_\textnormal{aux}^\textnormal{ps}$ whose corresponding representation is absolutely irreducible.
	
	1) $\operatorname{dim}(R_\textnormal{aux}^\textnormal{ps})_\mathfrak{p}\geq2[F:\mathbb{Q}]$.
	
	2) We have $c((R_\textnormal{aux}^\textnormal{ps})_\mathfrak{p})\geq2[F:\mathbb{Q}]-1$.
\end{prop}

\begin{proof}
	This is the argument of \cite[Lemmas 7.1.2 and 7.4.6]{pan2022fontaine}; we recall it since it applies under the present assumptions. After a finite extension of $k(\mathfrak{p})$, choose a stable lattice in the representation corresponding to $\mathfrak{p}$ whose reduction $\bar{\rho}_b$ is non-split, and let $R_b$ be its universal auxiliary deformation ring. If $\mathfrak{p}_b$ is the corresponding one-dimensional prime of $R_b$, Proposition \ref{pan}(3) gives
	$$
		\widehat{(R_\textnormal{aux}^\textnormal{ps})_{\mathfrak{p}}}
		\simeq\widehat{(R_b)_{\mathfrak{p}_b}}.
	$$
	Let $k$ be the residue field of the extended coefficient ring, and put
	$$h_i=\operatorname{dim}_{k}H^i(F_{\Sigma}/F,\operatorname{ad}^0\bar{\rho}_b),\qquad i=1,2.$$
	The tangent-obstruction argument gives $R_b\simeq\mathcal{O}[[x_1,\cdots,x_{h_1}]]/(f_1,\cdots,f_{h_2})$, and the global Euler characteristic formula gives $h_1-h_2=2[F:\mathbb{Q}]$. Localizing and completing therefore gives the first assertion.
	
	For the second assertion, let $S$ be the completion of $\mathcal{O}[[x_1,\cdots,x_{h_1}]]$ at the inverse image of $\mathfrak{p}_b$. Since $\mathfrak{p}_b$ is one-dimensional, $S$ is a regular local ring of dimension $h_1$, and $\widehat{(R_b)_{\mathfrak{p}_b}}\simeq S/(f_1,\cdots,f_{h_2})$. By \cite[Proposition 5.4.4]{pan2022fontaine},
	$$
		c\bigl(\widehat{(R_b)_{\mathfrak{p}_b}}\bigr)
		\geq h_1-h_2-1=2[F:\mathbb{Q}]-1.
	$$
	Finally, completion preserves dimension and does not increase connectedness dimension, which proves both assertions for $(R_\textnormal{aux}^\textnormal{ps})_{\mathfrak{p}}$.
\end{proof}

\begin{rem}
	Unlike the universal deformation ring $R_x^\textnormal{aux}$, the structure of the universal pseudo-deformation ring $R_\textnormal{aux}^\textnormal{ps}$ is less explicit, since the global Euler characteristic argument used in Proposition \ref{dimension} is not directly available. Nevertheless, suppose that $\mathfrak{p}$ is a one-dimensional prime of $R_\textnormal{aux}^\textnormal{ps}$ whose corresponding representation is absolutely irreducible. Then every irreducible component of $\operatorname{Spec}R_\textnormal{aux}^\textnormal{ps}$ containing $\mathfrak{p}$ has dimension at least $2[F:\mathbb{Q}]+1$, as an analogue of Proposition \ref{raux}.
	
	Indeed, put $A=(R_\textnormal{aux}^\textnormal{ps})_\mathfrak{p}$. If some irreducible component of $\operatorname{Spec}A$ had dimension less than $2[F:\mathbb{Q}]$, let $Z_1$ be the union of the components of dimension at least $2[F:\mathbb{Q}]$ and let $Z_2$ be the union of the remaining components. Proposition \ref{connect dim}(1) shows that $Z_1$ is non-empty, while our assumption makes $Z_2$ non-empty. By Proposition \ref{connect dim}(2), there exist components $C_1\subset Z_1$ and $C_2\subset Z_2$ such that $\dim(C_1\cap C_2)\geq2[F:\mathbb{Q}]-1$. Hence there is a prime $\mathfrak{q}\in C_1\cap C_2$ with $\dim A/\mathfrak{q}\geq2[F:\mathbb{Q}]-1$. Since $\dim C_2\leq2[F:\mathbb{Q}]-1$ and $A$ is catenary, the prime $\mathfrak{q}$ must be the generic point of $C_2$. This is impossible because the distinct irreducible components $C_1$ and $C_2$ have incomparable generic points. Thus, every irreducible component of $\operatorname{Spec}A$ has dimension at least $2[F:\mathbb{Q}]$.
	
	If $C=V(\mathfrak{P})$ is an irreducible component of $\operatorname{Spec}R_\textnormal{aux}^\textnormal{ps}$ containing $\mathfrak{p}$, then its localization at $\mathfrak{p}$ is an irreducible component of $\operatorname{Spec}A$. Since the complete local domain $R_\textnormal{aux}^\textnormal{ps}/\mathfrak{P}$ is catenary and $\dim R_\textnormal{aux}^\textnormal{ps}/\mathfrak{p}=1$, the dimension formula gives $\dim C=\dim C_\mathfrak{p}+1\geq2[F:\mathbb{Q}]+1$.
	
	It remains unclear whether every irreducible component of $R_\textnormal{aux}^\textnormal{ps}$ has dimension at least $2[F:\mathbb{Q}]+1$.
\end{rem}

\section{Pro-modularity}
In this section, we will prove the pro-modularity of the universal deformation ring. We fix a totally real field $F$ and a prime $p \ge 3$.

\subsection{Hecke algebras}
In this subsection, we introduce Hecke operators and Hecke algebras. The main reference is \cite[Section 3]{pan2022fontaine}.

Assume that $F$ has even degree over $\mathbb{Q}$ and that the prime $p$ splits completely in $F$.

Let $D$ be a quaternion algebra with centre $F$ which is ramified at all infinite places of $F$ and unramified at all places above $p$, and we fix isomorphisms $D \otimes_F F_v \cong M_2(F_v)$ for any $v$ where $D$ is unramified. We view $K_p = \prod_{v \mid p} \operatorname{GL}_2(\mathcal{O}_{F_v})$ and $D_p^\times =  \prod_{v \mid p} \operatorname{GL}_2(F_v)$ as subgroups of $(D \otimes_F \mathbb{A}_F^\infty)^\times$.

Let $A$ be a topological $\mathcal{O}$-algebra and $U = \prod_v U_v$ be an open compact subgroup of $(D \otimes_F \mathbb{A}_F^\infty)^\times$ such that $U_v \subseteq \operatorname{GL}_2(\mathcal{O}_{F_v})$ for $v \mid p$. We write $U^p= \prod_{v \nmid p} U_v$ (tame level) and $U_p =\prod_{v \mid p} U_v$. Let $\psi: (\mathbb{A}_F^\infty)^\times/F_+^\times \to A^\times$ be a continuous character, where $F_+$ is the set of totally positive elements in $F$. Let $\tau: U_p \to \operatorname{Aut}(W_\tau)$ be a continuous representation on a finite $A$-module $W_\tau$, and we also view $\tau$ as a representation of $U$ by projecting to $U_p$. 

Suppose that $\xi: U^p \to A^\times$ is a continuous smooth character such that $\psi|_{\prod_{v \nmid p}(\mathcal{O}_{F_v}^\times \cap U_v)}=\xi|_{\prod_{v \nmid p}(\mathcal{O}_{F_v}^\times \cap U_v)}$. Following \cite[Section 3.7.1]{pan2022fontaine}, we define $S_{\tau, \psi, \xi}(U,A)$ as the space of continuous functions $f: D^\times \setminus (D \otimes_F \mathbb{A}_F^\infty)^\times \to W_\tau$ such that for any $g \in (D \otimes_F \mathbb{A}_F^\infty)^\times$, $z \in (\mathbb{A}_F^\infty)^\times$, and $u=u^pu_p \in U$, we have $$
f(guz)= \psi(z)\xi(u^p)\tau(u_p^{-1})(f(g)).$$ 
If $\psi \mid_{U_p \cap \mathcal{O}_{F, p}^\times} = \tau^{-1} \mid_{U_p \cap \mathcal{O}_{F, p}^\times}$, then we have $$
S_{\tau, \psi, \xi}(U,A) \simeq \bigoplus_{i \in I} W_\tau ^{(t_i^{-1}D^{\times}t_i \cap U (\mathbb{A}_F^\infty)^\times)/F^\times},$$
where $I=D^\times \setminus (D \otimes_F \mathbb{A}_F^\infty)^\times/ U(\mathbb{A}_F^\infty)^\times$ and $\{t_i\}_{i \in I}$ is a set of representatives. We will simply write $S_{\psi, \xi}(U,A)$ if $\tau$ is the trivial action on $A$. We say that $U$ is \textit{sufficiently small} if $ (t_i^{-1}D^{\times}t_i \cap U (\mathbb{A}_F^\infty)^\times)/F^\times$ is trivial for all $i \in I$.

Let $\Sigma$ be a finite set of primes of $F$ containing all places above $p$ and places $v$ where either $D$ is ramified or $U_v$ is not a maximal compact subgroup. For any $v \notin \Sigma$, we define the Hecke operator $T_v \in \operatorname{End}(S_{\tau, \psi, \xi}(U,A))$ to be the double coset action $[U_v \begin{pmatrix}
	\pi_v & ~\\
	~ & 1
\end{pmatrix}U_v]$. Precisely, if we write $U_v \begin{pmatrix}
\pi_v & ~\\
~ & 1
\end{pmatrix}U_v  = \coprod_i \gamma_i U_v $, then $$
(T_v \cdot f) (g) = \sum_i f(g\gamma_i), ~~f \in S_{\tau, \psi, \xi}(U,A). $$

We define the Hecke algebra $\mathbb{T}_{\tau, \psi, \xi} ^\Sigma (U,A) \subseteq \operatorname{End}(S_{\tau, \psi, \xi}(U,A))$ to be the $A$-subalgebra generated by all $T_v$ for $v \notin \Sigma$. This is a finite commutative $A$-algebra. We will write $\mathbb{T}_{\psi, \xi}^\Sigma (U,A)$ if $\tau$ is the trivial action on $A$. Suppose that $A= \mathcal{O}/ \pi^n$, $U$ is sufficiently small, and $\psi|_{U \cap (\mathbb{A}_F^\infty)^\times}$ is trivial modulo $\pi^n$. Then $\mathbb{T}_{\psi, \xi}^\Sigma (U, \mathcal{O}/ \pi^n)$ is independent of $\Sigma$, and we may drop $\Sigma$ from the notation; see \cite[Sections 3.3.3 and 3.7.2]{pan2022fontaine}.

\begin{defn} (Big Hecke algebra) \label{big}
	Let $U^p$ be a tame level and $\psi: (\mathbb{A}_F^\infty)^\times/ (U^p \cap (\mathbb{A}_F^\infty)^\times)F_+^\times \to \mathcal{O}^\times$ be a continuous character. Following \cite[Definition 3.3.1 and Section 3.7.2]{pan2022fontaine}, we define the \textit{big Hecke algebra} $$
	\mathbb{T}_{\psi, \xi}(U^p) = \varprojlim_{(n, U_p) \in \mathcal{I}} \mathbb{T}_{\psi, \xi}(U^pU_p, \mathcal{O}/\pi^n), $$
	where $\mathcal{I}$ is the set of pairs $(n, U_p)$ with $U_p \subseteq K_p$ an open compact subgroup and $n$ a positive integer such that $ \psi \mid_{U_p \cap \mathcal{O}_{F,p}^\times} \equiv 1~ \operatorname{mod}~\pi^n$.
	
\end{defn}

\subsection{Deformation rings and Hecke algebras}\label{sec 2}
In this subsection, we collect some results about deformation rings and Hecke algebras, most of which are from \cite{pan2022fontaine}.

Assume $F$ is an abelian totally real number field of even degree over $\mathbb{Q}$ in which $p$ splits completely. As in Section \ref{mis}, we assume that $p \mid \operatorname{Nm}(v)-1$ for all $v \in \Sigma \setminus \Sigma_p$. Let $\xi_v: k(v)^\times \to \mathcal{O}^\times$ be characters of $p$-power order for $v \in \Sigma \setminus \Sigma_p$, which we also view as characters of $I_v$ by local class field theory. Let $\chi: \operatorname{Gal}(F_{\Sigma}/F) \to \mathcal{O}^{\times}$ be a continuous totally odd character such that 

1) $\chi$ is unramified outside $ \Sigma_p$,

2) $\chi(\operatorname{Frob}_v) \equiv 1~\operatorname{mod}~\pi$ for $v \in \Sigma \setminus \Sigma_p$.

Let $D$ be the quaternion algebra over $F$ ramified exactly at all infinite places, and fix an isomorphism $D \otimes_F \mathbb{A}_F^\infty\simeq M_2(\mathbb{A}_F^\infty)$. Let $\psi=\chi \epsilon$ be a character of $ (\mathbb{A}_F^\infty)^\times/ F_+^\times$ via global class field theory. Let $U^p=\prod_{v \nmid p} U_v$ be a tame level such that $U_v= \operatorname{GL}_2(\mathcal{O}_{F_v})$ if $v \notin \Sigma$ and $U_v = \operatorname{Iw}_v := \{g \equiv \begin{pmatrix}
	* & *\\
	0 & *
\end{pmatrix}
~\operatorname{mod}~\pi_v, g \in \operatorname{GL}_2(\mathcal{O}_{F_v})\}$ otherwise. For any $v \in \Sigma \setminus \Sigma_p$, the map $\begin{pmatrix}
	a & b\\
	c & d
\end{pmatrix} \to \xi_v (a/d~\operatorname{mod}~\pi_v)$ defines a character of $U_v$, and the product of the $\xi_v$ defines a character $\xi$ of $U^p$ by projection to $\prod_{v \in  \Sigma \setminus \Sigma_p} U_v$. We then define $\mathbb{T}_{\xi} := \mathbb{T}_{\psi, \xi}(U^p)$ as in the previous subsection.

For the statements in this subsection, we temporarily suppose that $\bar{\rho}_0= 1 \oplus \bar{\chi}$ is modular, namely, that $\pi$ and the elements $T_v-(1+\chi(\operatorname{Frob}_v))$, $v \notin \Sigma$, generate a maximal ideal $\mathfrak{m}_{\xi}$ of $ \mathbb{T}_{\xi}$. This temporary assumption will be removed in Subsection \ref{sec pro-modularity}.

\begin{thm} \label{dimofT}
	If $p=3$, assume further that $\bar{\chi}|_{G_{F_v}}\neq\omega_p^{\pm1}$ for every $v\mid p$.
	Each irreducible component of $(\mathbb{T}_{\xi})_{\mathfrak{m}_{\xi}} $ is of characteristic zero and of dimension at least $1+2[F:\mathbb{Q}]$.
\end{thm}

\begin{proof}
	This follows from \cite[Theorem 3.6.1 and Section 3.7.2]{pan2022fontaine}.
\end{proof}
	
	\begin{rem}
		In some cases, the big Hecke algebra is equidimensional of dimension $1+2[F:\mathbb{Q}]$. See the proof of \cite[Theorem 3.6.1]{pan2022fontaine} and \cite[Remark 3.6.10]{pan2022fontaine}.
	\end{rem}
	
By \cite[Sections 3.3.4, 3.7.2, and 4.1.3]{pan2022fontaine}, there is a natural pseudo-deformation $T_{\mathfrak{m}_{\xi}} : \operatorname{Gal}(F_{\Sigma}/F) \to (\mathbb{T}_{\xi})_{\mathfrak{m}_{\xi}}$ with determinant $\psi\epsilon^{-1}=\chi$ sending $\operatorname{Frob}_v$ to the Hecke operator $T_v$ for $v \notin \Sigma$. Since these Hecke operators topologically generate $(\mathbb{T}_{\xi})_{\mathfrak{m}_{\xi}}$, the universal property gives a natural surjection $ R_\textnormal{aux}^\textnormal{ps} \twoheadrightarrow (\mathbb{T}_{\xi})_{\mathfrak{m}_{\xi}}$. Local-global compatibility at $v \in \Sigma \setminus \Sigma_p$ gives $T_{\mathfrak{m}_{\xi}}|_{I_v}=\xi_v+\xi_v^{-1}$, so this surjection factors through a natural surjection $R^{\textnormal{ps}, \{\xi_v\}} \twoheadrightarrow (\mathbb{T}_{\xi})_{\mathfrak{m}_{\xi}}$.

For any $\mathfrak{p} \in \operatorname{Spec} (\mathbb{T}_{\xi})_{\mathfrak{m}_{\xi}} $, we can define a two-dimensional semi-simple representation $\rho(\mathfrak{p}): \operatorname{Gal}(F_{\Sigma}/F) \to \operatorname{GL}_2(k(\mathfrak{p}))$ with trace $T_{\mathfrak{m}_{\xi}}~\operatorname{mod}~\mathfrak{p}$ in the sense of Definition \ref{construction}.

The following definitions are from \cite[Section 4.1]{pan2022fontaine}.

\begin{defn}\label{nice}
	(1) We say that a prime  of $ R^{\textnormal{ps}, \{\xi_v\}}$ is \textit{pro-modular} if it comes from a prime of $(\mathbb{T}_{\xi})_{\mathfrak{m}_{\xi}} $. We say that a prime of $ R_\textnormal{aux}^\textnormal{ps}$ is \textit{pro-modular} if it comes from a pro-modular prime of some $ R^{\textnormal{ps}, \{\xi_v\}}$.
	
	(2) Let $\mathfrak{q} $ be a prime of $ (\mathbb{T}_{\xi})_{\mathfrak{m}_{\xi}} $ and let $A$ be the normalization of $(\mathbb{T}_{\xi})_{\mathfrak{m}_{\xi}} / \mathfrak{q} $ in $k(\mathfrak{q})$. We say that $\mathfrak{q} $ is a \textit{nice} prime if $\mathfrak{q} $ contains $p$ and $\operatorname{dim} (\mathbb{T}_{\xi})_{\mathfrak{m}_{\xi}} / \mathfrak{q}=1$ and there exists a two-dimensional representation $\rho(\mathfrak{q})^o : \operatorname{Gal}(F_{\Sigma}/F) \to \operatorname{GL}_2(A)$ satisfying:
	
	\quad 1) $\rho(\mathfrak{q})^o \otimes k(\mathfrak{q}) \cong \rho(\mathfrak{q})$ is irreducible.
	
	\quad 2) The $\operatorname{mod}~\mathfrak{m}_A$ reduction $\bar{\rho}_b$ of $\rho(\mathfrak{q})^o$ is a non-split extension and has the form $ \bar{\rho}_b(g)=\begin{pmatrix}
		* & *\\
		0 & *
	\end{pmatrix}$, $g \in \operatorname{Gal}(F_{\Sigma}/F)$. Here $\mathfrak{m}_A $ is the maximal ideal of $A$.
	
	\quad 3) (dihedral condition) If $\rho(\mathfrak{q}) $ is dihedral, namely isomorphic to $\operatorname{Ind}_{G_L}^{G_F} \theta$ for some quadratic extension $L$ of $F$ and continuous character $\theta: G_L \to k(\mathfrak{q})^\times$, then $L \cap F(\zeta_p) = F$,
	where $\zeta_p $ is a primitive $p$-th root of unity.
	
		\quad 4) $\rho(\mathfrak{q})^o \mid_{G_{F_v}} = \bar{\rho}_b \mid_{G_{F_v}} $ for any $v \in \Sigma \setminus \Sigma_p$, where the right-hand side is viewed as a representation over $A$ through the coefficient-field embedding. In other words, $\rho(\mathfrak{q})^o|_{G_{F_v}}$ is the trivial lift of $\bar{\rho}_b|_{G_{F_v}}$.
	
	(3) We say that a prime of $ R^{\textnormal{ps}, \{\xi_v\}}$ is \textit{nice} if it comes from a nice prime of $(\mathbb{T}_{\xi})_{\mathfrak{m}_{\xi}} $.
	
\end{defn}

The following proposition gives some sufficient conditions for the dihedral condition in Definition \ref{nice}(2).

\begin{prop}\label{equiv}
	Let $\mathfrak{q}$ be a prime ideal of $ R^{\textnormal{ps}, \{\xi_v\}}$ containing $p$ such that $ R^{\textnormal{ps}, \{\xi_v\}}/\mathfrak{q}$ is one-dimensional. Suppose $\rho(\mathfrak{q})$ is irreducible, and choose a stable lattice $\rho(\mathfrak{q})^o$ whose reduction is non-split as in Definition \ref{nice}(2). Then the dihedral condition in Definition \ref{nice}(2) holds if one of the following conditions holds:
	
	1)  $\bar{\chi}$ is not quadratic. 
	
	2) $\bar{\chi}\mid_{G_{F_v}}$ is trivial for some $v \mid p$.
	
		3) There exists a place $v \mid p$ such that $\bar{\chi}\mid_{G_{F_v}}$ is not trivial and $\rho(\mathfrak{q})^o\mid_{G_{F_v}} \cong \begin{pmatrix}
			\chi_{v,1} & *\\
			0 & \chi_{v,2}
		\end{pmatrix}$ is reducible and $\chi_{v,1}/\chi_{v,2}$ is of infinite order. Since $\chi_{v,1}\chi_{v,2}=\bar{\chi}|_{G_{F_v}}$ has finite order, this is equivalent to $\chi_{v,1}$ being of infinite order.
\end{prop}

\begin{proof}
	See \cite[Lemma 4.1.6]{pan2022fontaine}.
\end{proof}

The following theorem is a key input to our pro-modularity argument.

\begin{thm} \label{R=T}   $(R_{\mathfrak{q}}=\mathbb{T}_{\mathfrak{q}})$
	Let $\mathfrak{q} $ be a nice prime of $(\mathbb{T}_{\xi})_{\mathfrak{m}_{\xi}} $ and $\mathfrak{q}^\textnormal{ps}= \mathfrak{q} \cap R^{\textnormal{ps}, \{\xi_v\}}$. If $p=3$, assume further that $\bar{\chi}|_{G_{F_v}}\neq\omega_p^{\pm1}$ for every $v\mid p$. Then the natural surjective map $ (R^{\textnormal{ps}, \{\xi_v\}})_{\mathfrak{q}^\textnormal{ps}} \twoheadrightarrow (\mathbb{T}_{\xi})_{\mathfrak{q}}$ has nilpotent kernel.
	
\end{thm}

\begin{proof}
	This is \cite[Theorem 4.1.4]{pan2022fontaine}.
\end{proof}

\begin{rem}\label{dim of nice}
	We briefly summarize some main steps of the proof of Theorem \ref{R=T}. Here we follow most notations as in \cite[Section 4]{pan2022fontaine}, and one can find more details there.
		
	Let $\mathfrak{q}$ be a nice prime of $(\mathbb{T}_{\xi})_{\mathfrak{m}_{\xi}}$. After a finite unramified extension of the coefficient ring, we may assume that the normalization $A$ in Definition \ref{nice} has residue field $\mathbb{F}$ and fix an isomorphism $A\simeq\mathbb{F}[[t]]$, as in \cite[Section 4.1.4]{pan2022fontaine}. Let $B$ be the topological closure of the $\mathbb{F}$-algebra generated by all the entries of $\rho(\mathfrak{q})^o(\operatorname{Gal}(F_{\Sigma}/F))$. By the Chebotarev density theorem, we may find a finite set of primes $T'$ disjoint from $\Sigma$ such
	that the entries of $\rho(\mathfrak{q})^o(\operatorname{Frob}_v), v \in T'$, topologically generate $B$. Let $P= T' \cup \Sigma$.
	
	Let $r=\operatorname{dim}_{k(\mathfrak{q})} H^1(G_{F,P},\operatorname{ad}^0 \rho(\mathfrak{q})(1))$ (the cardinality of Taylor-Wiles primes). Let $R_{\infty}^{\{\xi_v\}}= R_{\textnormal{loc}}^{\{\xi_v\}}[[x_1, ..., x_g]]$, where $ R_{\textnormal{loc}}^{\{\xi_v\}}$ is the completed tensor product of local deformation rings determined by primes in $P$ (see \cite[Definition 4.2.2]{pan2022fontaine}) and $g=r+|P|-[F:\mathbb{Q}]-1$. Let $S_\infty'=\mathcal{O}[[y_1, ... , y_{4|P|-1}, s_1', ... , s_r']]$, and $S_\infty=\mathcal{O}[[y_1, ... , y_{4|P|-1}, s_1, ... , s_r]]$ with an ideal $\mathfrak{a}_1=(y_1, ... , y_{4|P|-1}, s_1, ... , s_r)$ such that $ S_\infty$ is a finite free $S_\infty'$-algebra.
	
	Let $\textnormal{m}_\infty^{\{\xi_v\}}$ and $\textnormal{m}_0^{\{\xi_v\}}$ be the \textit{patched Hecke modules }(see \cite[Section 4.7]{pan2022fontaine}). Then $\textnormal{m}_\infty^{\{\xi_v\}} $ is a flat  $ S_\infty$-module and $\textnormal{m}_\infty^{\{\xi_v\}}/ \mathfrak{a}_1 \textnormal{m}_\infty^{\{\xi_v\}} \cong (\textnormal{m}_0^{\{\xi_v\}})^{\oplus 2^r}$. By Pan's local-global compatibility result \cite[Corollary 3.5.8]{pan2022fontaine}, $\textnormal{m}_0^{\{\xi_v\}}$ is a finitely generated faithful $\widehat{(\mathbb{T}_{\xi})_{\mathfrak{q}}}$-module. By Theorem \ref{dimofT}, we have $\operatorname{dim}_{\widehat{(\mathbb{T}_{\xi})_{\mathfrak{q}}}} (\textnormal{m}_0^{\{\xi_v\}})= \operatorname{dim} \widehat{(\mathbb{T}_{\xi})_{\mathfrak{q}}} \ge 2[F:\mathbb{Q}]$ since $\mathfrak{q}$ is of dimension $1$.
	
	Let $(R^{\{\xi_v\}})'$ be $\widehat{(R_{\infty}^{\{\xi_v\}})_{\mathfrak{q}_{\infty}^{\{\xi_v\}}}} \otimes_{S_\infty'} S_\infty$, where ${\mathfrak{q}}_{\infty}^{\{\xi_v\}} $ is the prime of $ R_{\infty}^{\{\xi_v\}}$ corresponding to the nice prime $\mathfrak{q}$.	We know that $\widehat{(\mathbb{T}_{\xi})_{\mathfrak{q}}} $ is a natural quotient of $(R^{\{\xi_v\}})'$ by mapping $ S_\infty$ to $S_\infty/ \mathfrak{a}_1=\mathcal{O}$. Note that $ y_1, ... , y_{4|P|-1}, s_1, ... , s_r \in \mathfrak{a}_1$ form a regular sequence of $\textnormal{m}_\infty^{\{\xi_v\}}$. Then we can conclude that $$\operatorname{dim}_{(R^{\{\xi_v\}})'}(\textnormal{m}_\infty^{\{\xi_v\}}) \ge 4|P|-1 +r +\operatorname{dim} \widehat{(\mathbb{T}_{\xi})_{\mathfrak{q}}} \ge 4|P|-1 +r+2[F:\mathbb{Q}].$$
	
	Using Taylor's Ihara avoidance trick as in \cite{PMIHES_2008__108__183_0}, one can show that  
	$\widehat{(R_{\infty}^{\{\xi_v\}})_{\mathfrak{q}_{\infty}^{\{\xi_v\}}}}$ is equidimensional of dimension $ 4|P|-1 +r+2[F:\mathbb{Q}]$ (see \cite[Section 4.2, Lemma 4.8.2]{pan2022fontaine}). As $(R^{\{\xi_v\}})'$ is finite free over $\widehat{(R_{\infty}^{\{\xi_v\}})_{\mathfrak{q}_{\infty}^{\{\xi_v\}}}}$, we have $$\operatorname{dim}_{(R^{\{\xi_v\}})'}(\textnormal{m}_\infty^{\{\xi_v\}})=\operatorname{dim}_{\widehat{(R_{\infty}^{\{\xi_v\}})_{\mathfrak{q}_{\infty}^{\{\xi_v\}}}}}(\textnormal{m}_\infty^{\{\xi_v\}})=4|P|-1 +r+2[F:\mathbb{Q}]= \operatorname{dim} \widehat{(R_{\infty}^{\{\xi_v\}})_{\mathfrak{q}_{\infty}^{\{\xi_v\}}}}.$$
	
	If the characters in $\{\xi_v\}$ are all non-trivial, then \cite[Proposition 4.2.3]{pan2022fontaine} shows that $\widehat{(R_{\infty}^{\{\xi_v\}})_{\mathfrak{q}_{\infty}^{\{\xi_v\}}}}$ is irreducible. The previous equations imply that $\textnormal{m}_\infty^{\{\xi_v\}} $ has full support on this ring. By \cite[Lemma 4.8.4]{pan2022fontaine}, Theorem \ref{R=T} follows in this case. As in the proof of \cite[Theorem 4.1]{PMIHES_2008__108__183_0}, the general case follows from Taylor's trick.
	
	From Theorem \ref{dimofT}, \cite[Corollary 12.5]{eisenbud2013commutative}, and discussions above, we can also conclude that for a nice prime $\mathfrak{q}$ of $(\mathbb{T}_{\xi})_{\mathfrak{m}_{\xi}}$, we have $\operatorname{dim} \widehat{(\mathbb{T}_{\xi})_{\mathfrak{q}}} =\operatorname{dim} (\mathbb{T}_{\xi})_{\mathfrak{q}}= 2[F:\mathbb{Q}]$. In other words, if an irreducible component of the big Hecke algebra contains a nice prime, then it is of dimension $1+2[F:\mathbb{Q}]$.
\end{rem}

\begin{rem}\label{remark}
	From the definitions and the "$R=\mathbb{T}$" theorem above, we can conclude the following two facts.
	
	1) If a minimal prime $\mathfrak{P} \in \operatorname{Spec} R^{\textnormal{ps}, \{\xi_v\}}$ is pro-modular, then every prime lying on the irreducible component defined by $\mathfrak{P}$ is pro-modular. We also say that an irreducible component of $\operatorname{Spec} R^{\textnormal{ps}, \{\xi_v\}}$ is pro-modular if its generic point is pro-modular.
	
	2) If an irreducible component of $\operatorname{Spec} R^{\textnormal{ps}, \{\xi_v\}}$ contains a nice prime, then it is pro-modular.
\end{rem}

\subsection{The ordinary case}\label{sec 3}
In this subsection, we recall some results on the deformation rings and Hecke algebras in the ordinary case.

Assume that $F$ is an abelian totally real number field over $\mathbb{Q}$ in which $p$ is unramified. Let $\Sigma$ be a finite set of primes of $F$ containing all places above $p$. Let $\chi: \operatorname{Gal}(F_{\Sigma}/F) \to \mathcal{O}^{\times}$ be a continuous totally odd character such that 

1) $\bar{\chi}$ can be extended to a character of $G_{\mathbb{Q}}$.

2) $\bar{\chi} \mid_{G_{F_v}}$ is not trivial for any $v \mid p$.
	
3) $\chi \mid_{G_{F_v}}$ is de Rham for any $v \mid p$.

Let $(R^{\textnormal{ps,ord}}, T^{\textnormal{ord}})$ be the universal pseudo-deformation ring of $\operatorname{tr}(\bar{\rho}_0)$ of auxiliary type such that $ T^{\textnormal{ord}}\mid_{G_{F_v}}$ is reducible for any $v \mid p$, i.e. $y(\sigma, \tau)=0$ for any $\sigma, \tau \in G_{F_v}$, $v \mid p$.

Since $\bar{\chi} \mid_{G_{F_v}}$ is not trivial, we have $T^{\textnormal{ord}}\mid_{G_{F_v}}= \psi_{v,1}^{\textnormal{ord}}+\psi_{v,2}^{\textnormal{ord}}$ for unique characters $ \psi_{v,1}^{\textnormal{ord}}, \psi_{v,2}^{\textnormal{ord}}: G_{F_v} \to (R^{\textnormal{ps,ord}})^\times$ lifting $\mathbf{1}$ and $\bar{\chi}\mid_{G_{F_v}}$, respectively. Hence, via local class field theory, $\psi_{v,1}^{\textnormal{ord}}\mid_{I_v}$ induces a homomorphism $\mathcal{O}[[\mathcal{O}_{F_v}^\times(p)]] \to R^{\textnormal{ps,ord}}$ for any $v \mid p$, where $\mathcal{O}_{F_v}^\times(p) $ is the $p$-adic completion of $ \mathcal{O}_{F_v}^\times$.

\begin{defn}\cite[Definition 5.3.1]{pan2022fontaine}
	Let $T: \operatorname{Gal}(F_\Sigma/F) \to R$ be a two-dimensional pseudo-representation over some ring $R$ such that for some place $v$, $T \mid_{G_{F_v}} = \psi_1 + \psi_2$ is a sum of two characters. We say $T$ is $\psi_1$-\textit{ordinary} if for any $\sigma, \tau \in G_{F_v}$, $\eta \in \operatorname{Gal}(F_\Sigma/F)$, $$
	T(\sigma \tau \eta)-\psi_1(\sigma)T(\tau \eta)-\psi_2(\tau)T(\sigma\eta)+\psi_1(\sigma)\psi_2(\tau)T(\eta)=0.$$
\end{defn}

\begin{prop}\label{psi1}
	Let $\rho : \operatorname{Gal}(F_\Sigma/F) \to  \operatorname{GL}_2(R)$ be a two-dimensional irreducible representation over some ring $R$ with trace $T$ such that $T \mid_{G_{F_v}} = \psi_1 + \psi_2$, a sum of two characters. Suppose $\rho\mid_{G_{F_v}}$ has the form $\begin{pmatrix}
		\psi_{1} & *\\
		0 & \psi_{2}
	\end{pmatrix}$. Then $T$ is $\psi_1$-ordinary. Conversely, if $R$ is a field and $T$ is $\psi_1$-ordinary, then after possibly enlarging $R$,
	we have $\rho\mid_{G_{F_v}} \cong \begin{pmatrix}
		\psi_{1} & *\\
		0 & \psi_{2}
	\end{pmatrix}$.
\end{prop}

\begin{proof}
	See \cite[Lemma 5.3.2]{pan2022fontaine}.
\end{proof}

Define the \textit{Iwasawa algebra} $$
\Lambda_F := \widehat{\bigotimes}_{v \mid p} \mathcal{O}[[\mathcal{O}_{F_v}^\times(p)]], $$
where the completed tensor product is taken over $\mathcal{O}$. The characters $ \psi_{v,1}^{\textnormal{ord}}$ induce a natural map $\Lambda_F \to R^{\textnormal{ps,ord}} $. The following result is the ordinary Fontaine-Mazur conjecture proved in \cite{skinner1999residually} and \cite[Section 5]{pan2022fontaine}, which is used to produce pro-modular primes in our setting.

\begin{thm}[ordinary Fontaine-Mazur]\label{ordinary}
	(1) $ R^{\textnormal{ps,ord}}$ is a finite $ \Lambda_F$-algebra. 
	
	(2) For any maximal ideal $\mathfrak{p}$ of $R^{\textnormal{ps,ord}}[\frac{1}{p}]$, we denote the associated semi-simple representation $\operatorname{Gal}(F_{\Sigma}/F) \to \operatorname{GL}_2(k(\mathfrak{p}))$ by $\rho(\mathfrak{p})$. Assume that
	
	\quad i) $ \rho(\mathfrak{p})$ is irreducible.
	
	\quad ii) For any $v \mid p$, $\rho(\mathfrak{p}) \mid_{G_{F_v}} \cong  \begin{pmatrix}
		\chi_{v,1} & *\\
		0 & \chi_{v,2}
	\end{pmatrix}$ such that $\chi_{v,1} $ is de Rham and has strictly less Hodge-Tate number than $\chi_{v,2} $ for any embedding $F_v \hookrightarrow \overline{\mathbb{Q}}_p$.
	
		Then $ \rho(\mathfrak{p})$ comes from a twist of a Hilbert modular form.
\end{thm}

\begin{proof}
	See \cite[Theorem 5.1.1]{pan2022fontaine}.
\end{proof}

\begin{lem}\label{dense}
	Assume $[F:\mathbb{Q}]\geq2$, and let $C^{\textnormal{ord}}$ be an irreducible component of $\operatorname{Spec}R^{\textnormal{ps,ord}}$ of dimension at least $1+[F: \mathbb{Q}]$. Let $ C^{\textnormal{ord,aut}}$ be the set of regular de Rham primes in $C^{\textnormal{ord}}$. Then $C^{\textnormal{ord,aut}} $ is dense in $C^{\textnormal{ord}}$. Here a prime $\mathfrak{q}$ is called regular de Rham if
	
	\quad i) $p \notin \mathfrak{q}$ and $R^{\textnormal{ps,ord}}/\mathfrak{q} $ is one-dimensional.
	
		\quad ii) Let $\rho(\mathfrak{q})$ be the semi-simple representation associated to $T^{\textnormal{ord}}~\operatorname{mod}~\mathfrak{q}$ as in Definition \ref{construction}. Then $\rho(\mathfrak{q})$ is irreducible.
	
	\quad iii) For any $v \mid p$, $\rho(\mathfrak{q}) \mid_{G_{F_v}} \cong  \begin{pmatrix}
		\chi_{v,1} & *\\
		0 & \chi_{v,2}
	\end{pmatrix}$ such that $\chi_{v,1} $ is de Rham and has strictly less Hodge-Tate number than $\chi_{v,2} $ for any embedding $F_v \hookrightarrow \overline{\mathbb{Q}}_p$.
\end{lem}

\begin{proof}
	This is the argument of \cite[Corollary 7.2.3]{pan2022fontaine}. By Theorem \ref{ordinary}(1), the map $\Lambda_F\to R^{\textnormal{ps,ord}}$ is finite. Since $\Lambda_F$ is a domain of dimension $1+[F:\mathbb{Q}]$, the dimension assumption implies that $C^{\textnormal{ord}}$ has dimension exactly $1+[F:\mathbb{Q}]$ and that the induced finite map $C^{\textnormal{ord}}\to\operatorname{Spec}\Lambda_F$ is surjective. Moreover, Proposition \ref{red-pseudo}(1) and $[F:\mathbb{Q}]\geq2$ show that the globally reducible locus is a proper closed subset of $C^{\textnormal{ord}}$.
		
	For $v \mid p$, let $ \psi_{v,1}^{\textnormal{ord}}, \psi_{v,2}^{\textnormal{ord}}: G_{F_v} \to (R^{\textnormal{ps,ord}})^\times$ be the liftings of $\mathbf{1}$ and $\bar{\chi}\mid_{G_{F_v}}$, respectively. On the dense open subset where the associated global representation is irreducible, Proposition \ref{psi1} and the irreducibility of $C^{\textnormal{ord}}$ give an index $n_v \in \{1,2\}$, independent of $\mathfrak{q}$, such that
	$$
		\rho(\mathfrak{q}) \mid_{G_{F_v}} \cong
		\begin{pmatrix}
			\psi_{v,n_v}^{\textnormal{ord}}~\operatorname{mod}~\mathfrak{q}& *\\
			0 & \psi_{v,3-n_v}^{\textnormal{ord}}~\operatorname{mod}~\mathfrak{q}
		\end{pmatrix}.
	$$
	The map from $\Lambda_F$ is induced by the restrictions $\psi_{v,1}^{\textnormal{ord}}|_{I_v}$. Locally algebraic characters whose weights satisfy the strict inequalities determined by the indices $n_v$ form a Zariski-dense subset of $\operatorname{Spec}\Lambda_F$. Since the map from $C^{\textnormal{ord}}$ is finite and surjective and the reducible locus is proper, the corresponding regular de Rham primes are dense in $C^{\textnormal{ord}}$.
\end{proof}


	

\subsection{The pro-modularity of universal deformation rings}\label{sec pro-modularity}
In this subsection, we will prove our pro-modularity statements.

Keep the notations and assumptions on the finite set $\Sigma$ and the character $\chi$ in the previous two subsections. In this subsection, we further assume that $\bar{\chi}\mid_{G_{F_v}} \ne \omega_p^{\pm 1}$ for any $v\mid p$, and $[F:\mathbb{Q}] \ge 7|\Sigma \setminus \Sigma_p|+2.$

We do not retain the temporary modularity assumption made in Subsection \ref{sec 2}. It will follow in the proof below from the ordinary finiteness and density results once the relevant system $\{\xi_v\}$ has been fixed.

For a system $\{\xi_v\}$, put
$$
	R_x^{\textnormal{aux},\{\xi_v\}}:=R_x^{\textnormal{aux}}\widehat{\otimes}_{R_\textnormal{aux}^\textnormal{ps}}R^{\textnormal{ps},\{\xi_v\}}.
$$
Equivalently, this is the quotient of $R_x^\textnormal{aux}$ imposing the local pseudo-deformation conditions determined by $\{\xi_v\}$.

\begin{defn}\label{pro-modular}
	We say that a prime of $ R_x^{\textnormal{aux}, \{\xi_v\}}$ is pro-modular if its inverse image under the natural map $ R^{\textnormal{ps}, \{\xi_v\}} \to R_x^{\textnormal{aux}, \{\xi_v\}}$ is pro-modular. We say that a prime of $ R_x^\textnormal{aux}$ is pro-modular if it is the inverse image of some pro-modular prime of some $ R_x^{\textnormal{aux}, \{\xi_v\}}$ via the natural surjection $ R_x^\textnormal{aux} \twoheadrightarrow R_x^{\textnormal{aux}, \{\xi_v\}}$. 
\end{defn}

\begin{thm}\label{pro}
	Every irreducible component of $R_\textnormal{aux}^\textnormal{ps} $ of dimension at least $1+2[F:\mathbb{Q}]$ is pro-modular and has dimension exactly $1+2[F:\mathbb{Q}]$.
\end{thm}

\begin{proof}
	Let $C$ be an irreducible component of $R_\textnormal{aux}^\textnormal{ps} $ of dimension at least $1+2[F:\mathbb{Q}]$, and let $\mathfrak{P}$ be its generic point. By Proposition \ref{RX1}(1), after a finite extension of the coefficient field if necessary, the component $C$ belongs to a system $\{\xi_v\}$, which we now fix. We claim that $\mathfrak{P}$ is pro-modular.
	
	For the reader's convenience, we briefly introduce our strategy. By Remark \ref{remark}, we just need to find a nice prime containing $\mathfrak{P}$. In other words, we need to find a pro-modular prime satisfying the conditions listed in Definition \ref{nice}(2). Hence, our proof can be divided into two parts. In the first part, following \cite[Section 7]{pan2022fontaine}, we use Pan's results given in the previous section to find enough pro-modular primes from ordinary points. (There is a similar step in the proof of Deo's big $R=\mathbb{T}$ theorem; see \cite[Proposition 3.4]{deo2023density}.) In the second part, we use Corollary \ref{innovation'} to find a nice prime, which is a main innovation of this paper.\\
	
    Write $C^{\textnormal{ord}}=\operatorname{Spec} R^{\textnormal{ps,ord}} \cap C$.
	
	\begin{lem}\label{cord}
		$\operatorname{dim} C^{\textnormal{ord}} \ge [F:\mathbb{Q}]+1$.
	\end{lem} 
	
	\begin{proof}
		For any $v \mid p$, let $R_v^\textnormal{ps}$ be the universal pseudo-deformation ring of the pseudo-representation of $G_{F_v} $ lifting $1+\bar{\chi}\mid_{G_{F_v}}$ of auxiliary type, i.e. of determinant $\chi \mid_{G_{F_v}}$. Let $R_v^{\textnormal{ps,ord}}$ be the quotient of $R_v^\textnormal{ps}$ parametrizing all reducible liftings. As $\bar{\chi} \ne 1, \omega_p^{\pm 1}$, we know that the kernel of $ R_v^\textnormal{ps} \twoheadrightarrow R_v^{\textnormal{ps,ord}}$ is a principal ideal by \cite[Corollary B.20]{pavskunas2013image}. Write $R_p^\textnormal{ps}=\widehat{\bigotimes}_{\mathcal{O}} R_v^\textnormal{ps}$ and $R_p^{\textnormal{ps,ord}}=\widehat{\bigotimes}_{\mathcal{O}} R_v^{\textnormal{ps,ord}}$. Then by universal property, we get $R^{\textnormal{ps,ord}} =R_\textnormal{aux}^\textnormal{ps} \otimes_{R_p^\textnormal{ps}} R_p^{\textnormal{ps,ord}}$. Thus, $\operatorname{dim} C^{\textnormal{ord}} \ge \operatorname{dim} C -  |\Sigma_p | \ge 1+ 2[F:\mathbb{Q}]-[F:\mathbb{Q}]=[F:\mathbb{Q}]+1$.
	\end{proof}

	Now we can choose an irreducible component $C_1^{\textnormal{ord}} \subset C^{\textnormal{ord}}$ of $\operatorname{Spec} R^{\textnormal{ps,ord}}$ of dimension $1+[F: \mathbb{Q}]$, and we let $\mathfrak{Q}$ be the generic point of $C_1^{\textnormal{ord}}$. By Lemma \ref{dense} and Theorem \ref{ordinary}(2), this component contains ordinary automorphic points. Together with Proposition \ref{char} and \cite[Lemma 4.8.5]{pan2022fontaine}, this supplies the maximal ideal $\mathfrak{m}_{\xi}$ required in Subsection \ref{sec 2}. Thus, the Hecke-theoretic results of that subsection apply without assuming in advance that $1\oplus\bar{\chi}$ is modular.
	
	For any $v \in \Sigma \setminus \Sigma_p$, we fix a lifting $\sigma_v $ of the generator of the $\mathbb{Z}_p$-quotient of $I_v$ and a lifting $\operatorname{Frob}_v$ of the Frobenius element. Write $S=\{\sigma_v, \operatorname{Frob}_v:~v \in \Sigma \setminus \Sigma_p\}$. Thus, $|S|=2|\Sigma \setminus \Sigma_p|$.
	
	Let $I$ be an ideal of $R_\textnormal{aux}^\textnormal{ps}/(\mathfrak{Q}, \pi)$ generated by $$\{a(\sigma_{v})-1, a(\operatorname{Frob}_{v})-1, d(\operatorname{Frob}_{v})-1,~ v\ \in \Sigma \setminus \Sigma_p\}.$$ 
	We consider a minimal prime $\mathfrak{u}$ of $I$. By Krull's principal ideal theorem, $(R_\textnormal{aux}^\textnormal{ps}/(\mathfrak{Q}, \pi))/\mathfrak{u}$ is of dimension at least $[F: \mathbb{Q}]-3|\Sigma \setminus \Sigma_p|$. For simplicity, we write such a domain as $R_0$.\\
	
	To find a nice prime, we consider the following steps. Their purpose is to arrange that, for every element of $S$, one of the corresponding rows or columns of $y$ vanishes, while the remaining non-zero rows or columns are proportional by roots of unity. This will allow us to choose a single stable lattice which has non-split reduction and at the same time satisfies the local trivial-lifting condition in Definition \ref{nice}(2).\\
	
	\textit{Step} I. We start from the CNL domain $R_0$.
	
	By Corollary \ref{innovation'}, we can find a partition of $S = S_1 \amalg S_2$, a positive integer $n>1$ prime to $p$ and a CNL domain $R_0'$ satisfying the following conditions.
	
	1) $R_0'$ is a quotient of $R_0$.
	
	2) For any $\theta \in S_1$, we have $y(\theta, \alpha)=0$ for any $\alpha \in \operatorname{Gal}(F_\Sigma/F)$ in $R_0'$. For any $\theta', \theta'' \in S_2$, we have $y(\theta', \alpha)^n=y(\theta'', \alpha)^n$ for any $\alpha \in \operatorname{Gal}(F_\Sigma/F)$ in $R_0'$.
	
	3) For any $\alpha \in \operatorname{Gal}(F_\Sigma/F), \theta \in S_2$, either $ y(\theta, \beta)=0$ for any $\beta \in \operatorname{Gal}(F_\Sigma/F)$ or $u(\alpha, \theta)$ is well-defined and integral over $R_0'$.
	
	4) We have $\operatorname{dim} R_0' \ge \operatorname{dim} R_0 - |S|$. If $S_2$ is not empty, then further $\operatorname{dim} R_0' \ge \operatorname{dim} R_0 - |S|+1$.\\
	
	\textit{Step} II. We start from the CNL domain $R_0'$.
	
	Applying the transpose form of Corollary~\ref{innovation'}, we can find a partition of $S = S_1' \amalg S_2'$, a positive integer $n'>1$ prime to $p$ and a CNL domain $R_0''$ satisfying the following conditions.
	
	1') $R_0''$ is a quotient of $R_0'$.
	
	2') For any $\theta \in S_1'$, we have $y(\alpha, \theta)=0$ for any $\alpha \in \operatorname{Gal}(F_\Sigma/F)$ in $R_0''$. For any $\theta', \theta'' \in S_2'$, we have $y(\alpha, \theta')^{n'}=y(\alpha, \theta'')^{n'}$ for any $\alpha \in \operatorname{Gal}(F_\Sigma/F)$ in $R_0''$.
	
	3') For any $\alpha \in \operatorname{Gal}(F_\Sigma/F), \theta \in S_2'$, either $ y(\beta, \theta)=0$ for any $\beta \in \operatorname{Gal}(F_\Sigma/F)$ or $u'(\alpha, \theta)$ is well-defined and integral over $R_0''$.
	
	4') We have $\operatorname{dim} R_0'' \ge \operatorname{dim} R_0' - |S|$. If $S_2'$ is not empty, then further $\operatorname{dim} R_0'' \ge \operatorname{dim} R_0' - |S|+1$.\\
	 
	 \textit{Step} III. We consider the domain $R_0''$.
	 
		 If both $S_2$ and $S_2'$ are not empty, then we fix an element $\theta_{i} \in S_2$ and an element $\theta_{j} \in S_2'$. Consider a minimal prime $\mathfrak{p}_{ij}$ over the ideal $(y(\theta_{i}, \theta_{j}))$ of $R_0''$ and let $R''=R_0''/\mathfrak{p}_{ij}$. By Proposition \ref{u}, either $y(\theta_{i}, \alpha)=0$ for every $\alpha \in \operatorname{Gal}(F_\Sigma/F)$ or $y(\beta, \theta_{j})=0$ for every $\beta \in \operatorname{Gal}(F_\Sigma/F)$. By Krull's principal ideal theorem, we have $\operatorname{dim} R'' \ge \operatorname{dim} R_0'' -1$.
	 
	 If not, we take $R''=R_0''$.\\
	 
	 In conclusion, one of the following three assertions holds for $R''$.
	 
	 a) $ y(\theta, \alpha) =y(\alpha, \theta)=0$ for all $\theta \in S$ and $\alpha \in \operatorname{Gal}(F_\Sigma/F)$.
	 
	 b) There exists an element $\theta_{i} \in S$ such that either $y(\theta, \alpha)=0$ or $y(\theta, \alpha)^{n}=y(\theta_i, \alpha)^{n}$ for all $\theta \in S$ and $\alpha \in \operatorname{Gal}(F_\Sigma/F)$, and $y(\alpha, \theta) =0 $ for all $\theta \in S$ and $\alpha \in \operatorname{Gal}(F_\Sigma/F)$.
	 
	 c) There exists an element $\theta_{j} \in S$ such that either $y(\alpha, \theta)=0$ or $y(\alpha, \theta)^{n'}=y(\alpha, \theta_j)^{n'}$ for all $\theta \in S$ and $\alpha \in \operatorname{Gal}(F_\Sigma/F)$, and $y(\theta, \alpha) = 0 $ for all $\theta \in S$ and $\alpha \in \operatorname{Gal}(F_\Sigma/F)$.
	 
	 In particular, we have $\operatorname{dim} R'' \ge [F: \mathbb{Q}]-7|\Sigma \setminus \Sigma_p| \ge 2$.
	 
	\begin{lem}
		There exists a nice prime containing $\mathfrak{P}$.
	\end{lem}
	
	\begin{proof}
		By Theorem \ref{ordinary}(2), every prime in $C_1^{\textnormal{ord,aut}}$ comes from a twist of a Hilbert modular form and is therefore pro-modular. The pro-modular locus associated with the fixed big Hecke algebra is closed, since it is the image of the closed immersion induced by the corresponding quotient of $R^{\textnormal{ps},\{\xi_v\}}$. Lemma \ref{dense} therefore implies that every prime of $C_1^{\textnormal{ord}}$ is pro-modular.
		
		The proof now has two stages. We first specialize to an irreducible one-dimensional pro-modular prime. We then choose a stable lattice for its Galois representation. The main point of the second stage is that conditions 2) and 4) of Definition \ref{nice}(2) must hold for the same lattice; the infinite-order ordinary character needed for condition 3) will be recorded before the lattice is chosen.
		
		By Corollary \ref{red-one}, we can choose an irreducible one-dimensional prime $\mathfrak{r}'$ of $R''$. Its inverse image contains $\mathfrak{Q}$, $\pi$, and hence $\mathfrak{P}$, and the preceding paragraph shows that $\mathfrak{r}'$ is pro-modular. Thus, it contains $p$ and has one-dimensional quotient, as required in Definition \ref{nice}(2). We claim that it is a nice prime.
		
		Put $G=\operatorname{Gal}(F_{\Sigma}/F)$, and let $A$ be the normalization of $R''/\mathfrak{r}'$. Since $R''/\mathfrak{r}'$ is a one-dimensional complete local domain of characteristic $p$, its normalization is a complete discrete valuation ring. We fix a coefficient field and write $A\cong\mathbb{F}'[[T]]$. Choose $\sigma_{\mathfrak{r}'},\tau_{\mathfrak{r}'}\in G$ as in Definition \ref{construction}, and write $D=y(\sigma_{\mathfrak{r}'},\tau_{\mathfrak{r}'})\ne0$. The associated representation is
		$$
		\rho(\mathfrak{r}')(\sigma)=
		\begin{pmatrix}
			a(\sigma)&\dfrac{y(\sigma,\tau_{\mathfrak{r}'})}{D}\\
			y(\sigma_{\mathfrak{r}'},\sigma)&d(\sigma)
		\end{pmatrix} \in \operatorname{GL}_2(A).
		$$
		All the elements $y(\alpha,\beta)$ lie in the maximal ideal of $A$, since the residual pseudo-representation is reducible. We take $c$ to be the complex conjugation fixed in Definition \ref{pseudo}; thus $\rho(\mathfrak{r}')(c)=\begin{pmatrix}1&0\\0&-1\end{pmatrix}$.
		
		The next argument is independent of the choice of lattice and supplies the ordinary input for the dihedral condition. The composite map $\Lambda_F\to R^{\textnormal{ps,ord}}\to R''/\mathfrak{r}'\hookrightarrow A$ is finite: the first map is finite by Theorem \ref{ordinary}(1), and the normalization is finite because a complete local ring is Nagata. If every character $\psi_{v,1}^{\textnormal{ord}}$ had finite order after specialization to $A$, then its restriction to the pro-$p$ group $\mathcal{O}_{F_v}^{\times}(p)$ would be trivial, since $A$ is a domain of characteristic $p$. The kernel of $\Lambda_F\to A$ would then be the maximal ideal of $\Lambda_F$, contradicting the finiteness of this map and $\dim A=1$. Hence $\psi_{v,1}^{\textnormal{ord}}$ has infinite order for some $v\mid p$. Since the product of the two specialized ordinary characters is the image of $\chi$, which is the finite character $\bar{\chi}$ in $A$, and $p>2$, their ratio has infinite order. Once a stable lattice with non-split reduction has been chosen below, Proposition \ref{equiv}(3) gives condition 3) of Definition \ref{nice}(2).
		
		For every $\theta\in S$, we have $a(\theta)=d(\theta)=1$. Indeed, these equalities for $\operatorname{Frob}_v$ follow from the definition of $R_0$, while for $\sigma_v$ we also use $T(\sigma_v)=\xi_v(\sigma_v)+\xi_v(\sigma_v)^{-1}=2$ in characteristic $p$. The alternatives a)--c) above now serve to control the choice of lattice. In case a), the local rows and columns already vanish; in case b), one non-zero row provides a scaling factor; and case c) is the transposed situation. In every case, our aim is to make a non-zero residual upper-right extension class visible without losing the coefficient-field description of the local matrices.
		
		In case a), $\rho(\mathfrak{r}')(\theta)=1$ for every $\theta\in S$. Since $\rho(\mathfrak{r}')$ is irreducible, some upper-right matrix coefficient is non-zero. Let $m\geq0$ be the minimum of the $T$-adic valuations of all its non-zero upper-right coefficients. Conjugating by $\begin{pmatrix}T^{-m}&0\\0&1\end{pmatrix}$ gives another representation in $\operatorname{GL}_2(A)$, for which some upper-right coefficient is a unit, while all lower-left coefficients still lie in the maximal ideal. The value at $c$ remains diagonal, and the local representations at $v\in\Sigma\setminus\Sigma_p$ remain trivial. Thus, as in \cite[Lemma 2.13]{skinner1999residually}, the reduction of this lattice is upper triangular and non-split: its upper-right extension class is non-zero, since a coboundary vanishing at $c$ must vanish identically.
		
		We next consider case b), where a non-zero row is used to normalize all the local upper-right entries. The following unit calculation is what makes this normalization an integral change of lattice. If the $\theta_i$-row vanishes, then all the rows indexed by elements of $S$ vanish, and this is case a). Otherwise, Proposition \ref{u} gives $y(\theta_i,\tau_{\mathfrak{r}'})\ne0$. By Step I, $u(\sigma_{\mathfrak{r}'},\theta_i)$ is integral over $R_0'$, and hence belongs to $A$. On the other hand,
		$$
		u(\theta_i,\sigma_{\mathfrak{r}'})=
		\frac{y(\theta_i,\tau_{\mathfrak{r}'})}{D}\in A,
		\qquad
		u(\sigma_{\mathfrak{r}'},\theta_i)u(\theta_i,\sigma_{\mathfrak{r}'})=1.
		$$
		Thus, both elements are units in $A$. Put $P=\begin{pmatrix}u(\sigma_{\mathfrak{r}'},\theta_i)&0\\0&1\end{pmatrix}$. Using identity (4) following Definition \ref{pseudo}, we obtain
		$$
		\rho_1(\sigma):=P\rho(\mathfrak{r}')(\sigma)P^{-1}=
		\begin{pmatrix}
			a(\sigma)&u(\sigma,\theta_i)\\
			y(\theta_i,\sigma)&d(\sigma)
		\end{pmatrix}.
		$$
		For $\theta\in S$, the lower-left entry is zero, and either $u(\theta,\theta_i)=0$ or $u(\theta,\theta_i)^n=1$. Since $(n,p)=1$, Hensel's lemma gives $u(\theta,\theta_i)\in\mathbb{F}'$. Thus,
		$$
		\rho_1(\theta)=\begin{pmatrix}1&\lambda_\theta\\0&1\end{pmatrix},
		\qquad \lambda_\theta\in\mathbb{F}',
		\qquad \lambda_{\theta_i}=1.
		$$
		The conjugation by $P$ has put all the local matrices over $\mathbb{F}'$. We still have to normalize the value at $c$ without changing them. For this purpose, set $Q=\begin{pmatrix}1&u(c,\theta_i)/2\\0&1\end{pmatrix}\in\operatorname{GL}_2(A)$. Since $y(c,c)=0$, a direct calculation gives $Q\rho_1(c)Q^{-1}=\begin{pmatrix}1&0\\ *&-1\end{pmatrix}$. Moreover, $Q$ commutes with every matrix $\rho_1(\theta)$, $\theta\in S$. Replacing $\rho_1$ by $Q\rho_1Q^{-1}$ therefore preserves the displayed local matrices. Its reduction is upper triangular, and it satisfies
		$$
		\overline{\rho_1}(c)=\begin{pmatrix}1&0\\0&-1\end{pmatrix},
		\qquad
		\overline{\rho_1}(\theta_i)=\begin{pmatrix}1&1\\0&1\end{pmatrix}.
		$$
		Consequently, the reduction is non-split. Indeed, if its upper-right extension class were a coboundary, its value at $c$ would force the corresponding scalar to be zero since $p>2$, contradicting its value $1$ at $\theta_i$.
		
		Finally, case c) is the transpose of case b). We include the details because both the integral realization and the sign of the value at $c$ are interchanged. Consider
		$$
		\rho(\mathfrak{r}')'(\sigma)=
		\begin{pmatrix}
			a(\sigma)&y(\sigma,\tau_{\mathfrak{r}'})\\
			\dfrac{y(\sigma_{\mathfrak{r}'},\sigma)}{D}&d(\sigma)
		\end{pmatrix}.
		$$
		If the $\theta_j$-column vanishes, then all the columns indexed by elements of $S$ vanish. Conjugating $\rho(\mathfrak{r}')'$ by $\begin{pmatrix}0&1\\1&0\end{pmatrix}$ reduces this to case a). Otherwise, Step II shows that $u'(\tau_{\mathfrak{r}'},\theta_j)$ is integral, while $u'(\theta_j,\tau_{\mathfrak{r}'})=y(\sigma_{\mathfrak{r}'},\theta_j)/D$ belongs to $A$ and their product is $1$. Thus, both are units of $A$. Put
		$$
		P'=\begin{pmatrix}0&u'(\tau_{\mathfrak{r}'},\theta_j)\\1&0\end{pmatrix},
		\qquad
		(P')^{-1}=\begin{pmatrix}0&1\\u'(\theta_j,\tau_{\mathfrak{r}'})&0\end{pmatrix}.
		$$
		Then
		$$
		\rho_1'(\sigma):=P'\rho(\mathfrak{r}')'(\sigma)(P')^{-1}=
		\begin{pmatrix}
			d(\sigma)&u'(\sigma,\theta_j)\\
			y(\sigma,\theta_j)&a(\sigma)
		\end{pmatrix}.
		$$
		For every $\theta\in S$, we have $\rho_1'(\theta)=\begin{pmatrix}1&\lambda_\theta'\\0&1\end{pmatrix}$ with $\lambda_\theta'\in\mathbb{F}'$ and $\lambda_{\theta_j}'=1$. Now set $Q'=\begin{pmatrix}1&-u'(c,\theta_j)/2\\0&1\end{pmatrix}$. As in case b), conjugation by $Q'$ normalizes the value at $c$ and commutes with all these local matrices. The resulting reduction is upper triangular and satisfies
		$$
		\overline{\rho_1'}(c)=\begin{pmatrix}-1&0\\0&1\end{pmatrix},
		\qquad
		\overline{\rho_1'}(\theta_j)=\begin{pmatrix}1&1\\0&1\end{pmatrix}.
		$$
		The same coboundary argument at $c$ shows that this reduction is non-split, while the local matrices are unchanged.
		
		We have now obtained condition 2) in all three cases. It remains to check that the same lattice, rather than a new choice, satisfies condition 4). In all three cases, the matrices of $\sigma_v$ and $\operatorname{Frob}_v$ belong to $\operatorname{GL}_2(\mathbb{F}')$. As in the proof of Proposition \ref{char}, the restriction to inertia factors through the $\mathbb{Z}_p$-quotient generated by $\sigma_v$; together with $\operatorname{Frob}_v$, this determines the restriction to $G_{F_v}$. Thus, the chosen lattice satisfies $\rho(\mathfrak{r}')^o|_{G_{F_v}}=\bar{\rho}_b|_{G_{F_v}}$ through the coefficient-field embedding for every $v\in\Sigma\setminus\Sigma_p$.
		
		Finally, we return to condition 3). For $v\mid p$, the intersection of the ordinary line in the generic fibre with the chosen lattice is a saturated rank-one $A$-submodule. Hence the local restriction is upper triangular over $A$ with the same two ordinary characters. The infinite-order ratio found above therefore allows us to apply Proposition \ref{equiv}(3). We have now verified conditions 1)--4) of Definition \ref{nice}(2), so $\mathfrak{r}'$ is nice.
	\end{proof}
	
	From the previous lemma and Remark \ref{remark}(2), we know that $C$ is pro-modular. Moreover, the proof of Lemma \ref{cord} gives
	$$
		\operatorname{dim} C^\textnormal{ord}\geq \operatorname{dim} C-[F:\mathbb{Q}],
	$$
	while Theorem \ref{ordinary}(1) gives $\operatorname{dim} C^\textnormal{ord}\leq1+[F:\mathbb{Q}]$. Hence $\operatorname{dim} C\leq1+2[F:\mathbb{Q}]$, and equality follows from the hypothesis on $C$.
	
\end{proof}

\begin{rem}\label{reason}
	1)  By Proposition \ref{connect dim}, for a one-dimensional prime $\mathfrak{p}$ of $R_\textnormal{aux}^\textnormal{ps} $, if its associated representation is absolutely irreducible, then at least one of the irreducible components of $R_\textnormal{aux}^\textnormal{ps} $ containing $\mathfrak{p}$ is of dimension at least $1+2[F:\mathbb{Q}]$. Hence $\mathfrak{p}$ is pro-modular.
	
	2) The exact dimension in Theorem \ref{pro} comes from the two inequalities $\operatorname{dim} C^\textnormal{ord}\geq\operatorname{dim}C-[F:\mathbb{Q}]$ and $\operatorname{dim} C^\textnormal{ord}\leq1+[F:\mathbb{Q}]$. In particular, the ordinary locus of $C$ has dimension exactly $1+[F:\mathbb{Q}]$.
	
	3) The construction of the maximal ideal $\mathfrak{m}_{\xi}$ in the proof of Theorem \ref{pro} explains why the initial assumption that $1 \oplus \bar{\chi}$ is modular is unnecessary.
\end{rem}

\begin{thm}\label{MINIMAL}
	 If we further assume $$[F:\mathbb{Q}] \ge \operatorname{max}\{1+\frac{1}{2}\operatorname{dim}_{\mathbb{F}}(H^1(F_{\Sigma}/F, \mathbb{F}(\bar{\chi}^{-1}))), 7|\Sigma \setminus \Sigma_p|+2\},$$ 
	then for any non-zero element $x \in H^1(F_{\Sigma}/F, \mathbb{F}(\bar{\chi}^{-1}))$, any prime ideal of $ R_x^\textnormal{aux}$ is pro-modular. Moreover, $ R_x^\textnormal{aux}$ is a local complete intersection ring of Krull dimension $1+2[F:\mathbb{Q}] $.
	
\end{thm}

\begin{proof}
	Put $d=[F:\mathbb{Q}]$ and $m=\operatorname{dim}_{\mathbb{F}}H^1(F_{\Sigma}/F,\mathbb{F}(\bar{\chi}^{-1}))$. Let $\mathfrak{q}$ be a prime of $R_x^\textnormal{aux}$, and choose a minimal prime $\mathfrak{P}\subseteq\mathfrak{q}$. Proposition \ref{raux} gives $\operatorname{dim}R_x^\textnormal{aux}/\mathfrak{P}\geq1+2d$. On the other hand, since $F$ is abelian, Proposition \ref{red} and the degree assumption give
	$$
		\operatorname{dim}R_x^\textnormal{aux,red}\leq1+m\leq2d-1.
	$$
	Consequently, the representation at the generic point $\mathfrak{P}$ is irreducible: otherwise $R_x^\textnormal{aux}/\mathfrak{P}$ would be a quotient of $R_x^\textnormal{aux,red}$, contradicting the two dimension bounds above. Let $\mathfrak{P}'$ and $\mathfrak{q}'$ be the inverse images of $\mathfrak{P}$ and $\mathfrak{q}$ under the natural map $f:R_\textnormal{aux}^\textnormal{ps}\to R_x^\textnormal{aux}$. Proposition \ref{pan}(2) shows that $\mathfrak{P}'$ is minimal and that
	$$
		\operatorname{dim}R_\textnormal{aux}^\textnormal{ps}/\mathfrak{P}'
		\geq\operatorname{dim}R_x^\textnormal{aux}/\mathfrak{P}\geq1+2d.
	$$
	By Proposition \ref{RX1}(1), this component belongs to a system $\{\xi_v\}$, and it is pro-modular by Theorem \ref{pro}. Since $\ker(R_\textnormal{aux}^\textnormal{ps}\to R^{\textnormal{ps},\{\xi_v\}})\subseteq\mathfrak{P}'\subseteq\mathfrak{q}'$, the prime $\mathfrak{q}$ descends to $R_x^{\textnormal{aux},\{\xi_v\}}$. Remark \ref{remark}(1) shows that its inverse image in $R^{\textnormal{ps},\{\xi_v\}}$ is pro-modular, so $\mathfrak{q}$ is pro-modular by Definition \ref{pro-modular}. This proves the first assertion for every prime of $R_x^\textnormal{aux}$.
	
	The same argument applies to every minimal prime $\mathfrak{P}$ of $R_x^\textnormal{aux}$. Theorem \ref{pro} and Proposition \ref{pan}(2) give
	$$
		1+2d\leq\operatorname{dim}R_x^\textnormal{aux}/\mathfrak{P}
		\leq\operatorname{dim}R_\textnormal{aux}^\textnormal{ps}/\mathfrak{P}'=1+2d.
	$$
	Thus, every irreducible component of $R_x^\textnormal{aux}$ has dimension $1+2d$. To finish, write the presentation of Proposition \ref{dimension} as
	$$
		R_x^\textnormal{aux}\simeq\mathcal{O}[[x_1,\ldots,x_g]]/(f_1,\ldots,f_r),
		\qquad r\leq h_2:=\operatorname{dim}_{\mathbb{F}}H^2(F_{\Sigma}/F,\operatorname{ad}^0\rho_x).
	$$
	The global Euler characteristic formula gives $g-h_2=2d$. Since all the irreducible components have dimension $1+2d$, the defining ideal has height $h_2$. Krull's height theorem gives $h_2\leq r$, so $r=h_2$. Hence the defining equations form a regular sequence in the regular local ring $\mathcal{O}[[x_1,\ldots,x_g]]$, and $R_x^\textnormal{aux}$ is a local complete intersection of dimension $1+2d$.
	
\end{proof}

\subsection{A special case}
Keep the assumptions on the finite set $\Sigma$ and the character $\chi$ from Subsections \ref{sec 2} and \ref{sec 3}. As in Theorem \ref{main}, we also assume that $\bar{\chi}|_{G_{F_v}}\notin\{\mathbf{1},\omega_p^{-1}\}$ for every $v\mid p$. In this subsection, we deal with the case $\bar{\chi}|_{G_{F_v}}=\omega_p$ for every $v\mid p$, and we further assume that $[F:\mathbb{Q}]\geq 7|\Sigma\setminus\Sigma_p|+4$.

Since $\bar{\chi}$ extends to $G_{\mathbb{Q}}$ and $p$ splits completely in $F$, the restrictions $\bar{\chi}|_{G_{F_v}}$ for $v\mid p$ are identified with one another. Thus, under the local assumption in Theorem \ref{main}, this is the only case not covered by Theorem \ref{pro}. Notice also that the assumptions of this subsection force $p>3$, since $\omega_3=\omega_3^{-1}$. In particular, $\omega_p\neq\omega_p^{-1}$, so the results of \cite[Section 7.4]{pan2022fontaine} used below apply. As in Subsection \ref{sec pro-modularity}, we do not retain the temporary modularity assumption from Subsection \ref{sec 2}. The argument is more involved than that of the previous subsection because Lemma \ref{cord} no longer holds.

We first record an auxiliary proposition in the more general setting of Subsections \ref{sec 2} and \ref{sec 3}. For this proposition only, we do not impose the special assumption $\bar{\chi}|_{G_{F_v}}=\omega_p$ made above. After its proof, we return to the standing assumptions of this subsection.

\begin{prop}\label{promodular to nice}
	Assume the hypotheses on $F$, $\Sigma$, and $\chi$ from Subsections \ref{sec 2} and \ref{sec 3}, apart from the temporary modularity assumption in Subsection \ref{sec 2}, and assume that $[F:\mathbb{Q}]\geq 7|\Sigma\setminus\Sigma_p|+4$. If $p=3$, assume further that $\bar{\chi}|_{G_{F_v}}\neq\omega_p^{\pm1}$ for every $v\mid p$. Let $\mathfrak{r}$ be a one-dimensional irreducible pro-modular prime of $R_\textnormal{aux}^\textnormal{ps} $. Then every irreducible component of $R_\textnormal{aux}^\textnormal{ps} $ containing $\mathfrak{r}$ contains a nice prime, and hence is pro-modular.
\end{prop}

\begin{proof}
		Let $Z$ be the set of all irreducible components of $(R_\textnormal{aux}^\textnormal{ps})_{\mathfrak{r}} $. Write $Z= Z_1 \amalg Z_2$, where $Z_1$ is the subset of $Z$ consisting of all pro-modular irreducible components. We first show that if $Z_1$ is not empty, then $Z_2$ is empty.
		
		If not, by Proposition \ref{connect dim}, there exist an irreducible component $Y \in Z_2$ and a pro-modular prime $\mathfrak{r}' \in Y$ such that $\dim R_\textnormal{aux}^\textnormal{ps}/\mathfrak{r}'\geq 2[F:\mathbb{Q}]-1$.
		
		Arguing as in the proof of Theorem \ref{pro}, we may find a prime $\mathfrak{r}'\subset \mathfrak{r}''$ satisfying the following conditions:
		\begin{itemize}
			\item $\varpi \in \mathfrak{r}''$.
			\item $\dim R_\textnormal{aux}^\textnormal{ps}/\mathfrak{r}''\geq 2[F:\mathbb{Q}]-7|\Sigma\setminus\Sigma_p|-2\geq [F:\mathbb{Q}]+2$.
			\item  For any irreducible one-dimensional prime $\mathfrak{t}$ containing $\mathfrak{r}''$, there exists a $2$-dimensional Galois representation $\rho(\mathfrak{t}) : G_{F, \Sigma} \to \operatorname{GL}_2(k(\mathfrak{t}))$ satisfying the conditions in (2) of Definition \ref{nice} except the third one.
		\end{itemize}
		
		If $\bar{\chi}$ is not quadratic, Corollary \ref{red-one}(1) gives a one-dimensional irreducible prime $\mathfrak{t}$ containing $\mathfrak{r}''$. Since $\mathfrak{r}'\subset\mathfrak{t}$, the prime $\mathfrak{t}$ is pro-modular. Proposition \ref{equiv}(1) verifies the dihedral condition, so the choice of $\mathfrak{r}''$ shows that $\mathfrak{t}$ is a nice prime lying on $Y$. Theorem \ref{R=T} and Remark \ref{remark}(2) then imply that $Y$ is pro-modular, contrary to the definition of $Z_2$.
		
		If $\bar{\chi}$ is quadratic, its splitting field $F(\bar{\chi})$ is an abelian CM field. Corollary \ref{red-one}(2) gives a one-dimensional prime $\mathfrak{t}$ containing $\mathfrak{r}''$ such that $\rho(\mathfrak{t})$ is irreducible and not dihedral. The prime $\mathfrak{t}$ is again pro-modular, and the choice of $\mathfrak{r}''$ shows that it is nice. Theorem \ref{R=T} gives the same contradiction.
		
		It remains to show that $Z_1$ is non-empty. Since $\mathfrak{r}$ is pro-modular, Theorem \ref{dimofT} gives a pro-modular prime $\mathfrak{r}_1\subset\mathfrak{r}$ such that $\dim R_\textnormal{aux}^\textnormal{ps}/\mathfrak{r}_1\geq 1+2[F:\mathbb{Q}]$. Applying the preceding construction to $\mathfrak{r}_1$, we find a nice prime containing $\mathfrak{r}_1$. Hence at least one irreducible component containing $\mathfrak{r}$ is pro-modular, and $Z_1$ is non-empty. We have therefore proved that every element of $Z$ is pro-modular.
		
		Finally, let $Y\in Z$. Its generic point is pro-modular and, by Theorem \ref{dimofT}, has dimension at least $1+2[F:\mathbb{Q}]$. Applying the same construction to this generic point produces a nice prime on $Y$. Thus every irreducible component containing $\mathfrak{r}$ contains a nice prime, as asserted.
		
\end{proof}

We now return to the standing assumptions of this subsection.

Suppose that $C$ is an irreducible component of $R_\textnormal{aux}^\textnormal{ps}$ of dimension at least $1+2[F:\mathbb{Q}]$, and put $C^\textnormal{ord}=C\cap\operatorname{Spec}R^\textnormal{ps,ord}$. For every $v\mid p$, \cite[Lemma 7.4.4]{pan2022fontaine} shows that the kernel of the natural surjection $R_v^\textnormal{ps}\twoheadrightarrow R_v^\textnormal{ps,ord}$ can be generated by two elements. Since $p$ splits completely in $F$, the kernel of $R_\textnormal{aux}^\textnormal{ps}\twoheadrightarrow R^\textnormal{ps,ord}$ can therefore be generated by at most $2[F:\mathbb{Q}]$ elements. Krull's principal ideal theorem gives $\dim C^\textnormal{ord}\geq 1$.

Choose a prime $\mathfrak{q}\in C^\textnormal{ord}$ such that $\dim R_\textnormal{aux}^\textnormal{ps}/\mathfrak{q}=1$. Enlarging $\mathcal{O}$ if necessary, we may assume that the normalization of $R_\textnormal{aux}^\textnormal{ps}/\mathfrak{q}$ is either $\mathcal{O}$ or isomorphic to $\mathbb{F}[[T]]$. There are three possibilities for the associated semi-simple representation $\rho(\mathfrak{q})$:
\begin{enumerate}
		\item $\rho(\mathfrak{q})$ is irreducible.
		\item $\rho(\mathfrak{q})\cong\psi_1\oplus\psi_2$ is reducible and $\psi_1/\psi_2$ is not of the form $\epsilon^{\pm1}\theta$, where $\theta$ is a finite order character of $G_F$. We call this case \textit{generic reducible}.
		\item After possibly interchanging $\psi_1$ and $\psi_2$, we have $\rho(\mathfrak{q})\cong\psi_1\oplus\psi_2$ and $\psi_1/\psi_2=\epsilon\theta$ for some finite order character $\theta$ of $G_F$. We call this case \textit{non-generic reducible}.
\end{enumerate}

\begin{thm}\label{special case}
	The irreducible component $C$ is pro-modular of dimension $1+2[F:\mathbb{Q}]$.
\end{thm}

\begin{proof}
		We divide the proof into the three cases above.
		
		\textit{1. The irreducible case}
		
		After fixing the ordinary orientation at every place above $p$, the proof of \cite[Lemma 7.4.7]{pan2022fontaine} gives $\dim (R^{\textnormal{ps,ord}})_{\mathfrak{q}}\geq [F:\mathbb{Q}]$. Here the vanishing used in that proof follows from $\bar{\chi}|_{G_{F_v}}=\omega_p\neq\omega_p^{-1}$. Thus, there is an irreducible component $D^\textnormal{ord}$ of $\operatorname{Spec}R^\textnormal{ps,ord}$ containing $\mathfrak{q}$ and having dimension at least $1+[F:\mathbb{Q}]$. By Lemma \ref{dense}, Theorem \ref{ordinary}(2), and the closedness argument used in the proof of Theorem \ref{pro}, every prime of $D^\textnormal{ord}$ is pro-modular. Hence $\mathfrak{q}$ is a one-dimensional irreducible pro-modular prime, and Proposition \ref{promodular to nice} shows that $C$ is pro-modular.
		
		\textit{2. The generic reducible case}
		
		For each $v\mid p$, let $\mathfrak{q}_v$ be the pull-back of $\mathfrak{q}$ to the local ordinary pseudo-deformation ring. The argument preceding \cite[Lemma 7.4.8]{pan2022fontaine} shows that $\mathfrak{q}_v$ is one-dimensional and that the local ratio of the two characters is not $\epsilon^{\pm1}$. Consequently, \cite[Lemma 7.4.8 and Corollary 7.4.9]{pan2022fontaine} give $\dim C^\textnormal{ord}\geq 1+[F:\mathbb{Q}]$. Choose an irreducible component $D^\textnormal{ord}\subset C^\textnormal{ord}$ of this dimension. Lemma \ref{dense} gives a one-dimensional regular de Rham prime $\mathfrak{r}\in D^\textnormal{ord}$; it is irreducible and pro-modular by Theorem \ref{ordinary}(2). Proposition \ref{promodular to nice} therefore implies that $C$ is pro-modular.
		
		\textit{3. The non-generic reducible case}
		
		In this case, $\mathfrak{q}$ is of characteristic zero; otherwise the two characters would have finite image and $\mathfrak{q}$ would be the maximal ideal. Enlarging $E$ if necessary, we may assume that $k(\mathfrak{q})=E$. By \cite[Lemma 7.4.10]{pan2022fontaine}, the natural restriction map
		\[
		H^1(G_{F,\Sigma},E(\psi_2/\psi_1))\longrightarrow\bigoplus_{v\mid p}H^1(G_{F_v},E(\psi_2/\psi_1))
		\]
		is an isomorphism; in particular, the source has dimension $[F:\mathbb{Q}]$.
		
		For every non-zero class $B\in\operatorname{Ext}^1_{E[G_{F,\Sigma}]}(\psi_1,\psi_2)\cong H^1(G_{F,\Sigma},E(\psi_2/\psi_1))$, let $\rho_B$ be the corresponding non-split extension of $\psi_1$ by $\psi_2$, written as $\rho_B=\begin{pmatrix}\psi_2&b_B\\0&\psi_1\end{pmatrix}$, and let $R_B$ be its universal deformation ring with fixed determinant $\chi$. We use the following consequences of \cite[Lemmas 7.4.11 and 7.4.12, Corollary 7.4.13]{pan2022fontaine}:
		\begin{enumerate}
			\item every irreducible component of $R_B$ has dimension at least $2[F:\mathbb{Q}]$, and the connectedness dimension of $R_B$ is at least $2[F:\mathbb{Q}]-1$;
			\item the reducible locus $\operatorname{Spec}R_B^\textnormal{red}$ has dimension at most $[F:\mathbb{Q}]+1$;
			\item the trace gives a natural map $R_\textnormal{aux}^\textnormal{ps}\to R_B$, and if $\mathfrak{Q}\notin\operatorname{Spec}R_B^\textnormal{red}$ and $\mathfrak{Q}^\textnormal{ps}=\mathfrak{Q}\cap R_\textnormal{aux}^\textnormal{ps}$, then $\dim R_\textnormal{aux}^\textnormal{ps}/\mathfrak{Q}^\textnormal{ps}\geq 1+\dim R_B/\mathfrak{Q}$;
			\item the map $\operatorname{Spec}R_B\to\operatorname{Spec}R_\textnormal{aux}^\textnormal{ps}$ sends minimal primes to minimal primes.
		\end{enumerate}
		Let $Z_B$ be the set of irreducible components of $R_\textnormal{aux}^\textnormal{ps}$ whose generic points lie in the image of $\operatorname{Spec}R_B\to\operatorname{Spec}R_\textnormal{aux}^\textnormal{ps}$. By (4), the union of the components in $Z_B$ contains the image of $\operatorname{Spec}R_B$.
		
		We need the following pro-modular variant of \cite[Corollary 7.4.15]{pan2022fontaine}.
		
		\textit{Propagation claim.} If one irreducible component in $Z_B$ is pro-modular, then every irreducible component in $Z_B$ is pro-modular.
		
		Indeed, suppose otherwise, and write the union of the components in $Z_B$ as $X_1\cup X_2$, where $X_1$ is the union of the pro-modular components and $X_2$ is the union of the remaining components. Let $\mathcal{X}_i$ be the union of the irreducible components of $\operatorname{Spec}R_B$ whose images are contained in $X_i$. The two sets $\mathcal{X}_1$ and $\mathcal{X}_2$ cover $\operatorname{Spec}R_B$. By (1), there is a prime $\mathfrak{Q}\in\mathcal{X}_1\cap\mathcal{X}_2$ such that $\dim R_B/\mathfrak{Q}\geq 2[F:\mathbb{Q}]-1$. Since $2[F:\mathbb{Q}]-1>[F:\mathbb{Q}]+1$, part (2) shows that $\mathfrak{Q}$ is not reducible. Hence, by (3), $\dim R_\textnormal{aux}^\textnormal{ps}/\mathfrak{Q}^\textnormal{ps}\geq 2[F:\mathbb{Q}]$, and $\mathfrak{Q}^\textnormal{ps}$ belongs to $X_1\cap X_2$. Proposition \ref{red-pseudo}(1) and Lemma \ref{Jacobon} now give a one-dimensional irreducible prime $\mathfrak{r}$ containing $\mathfrak{Q}^\textnormal{ps}$. Since $\mathfrak{r}$ lies on a component of $X_1$, it is pro-modular; since it also lies on a component of $X_2$, Proposition \ref{promodular to nice} gives a contradiction. This proves the claim.
		
		We now follow the two steps in the proof of \cite[Proposition 7.4.3, Section 7.4.5]{pan2022fontaine}. First, fix a place $v_0\mid p$. The restriction isomorphism above gives a non-zero class $B_0$ whose localization is zero at every $v\mid p$ other than $v_0$. The argument of \cite[proof of Proposition 7.4.3, Step (1), pp.~1152--1153]{pan2022fontaine}, applied together with the irreducible case proved above, produces a pro-modular irreducible component in $Z_{B_0}$. The propagation claim shows that every component in $Z_{B_0}$ is pro-modular.
		
		Second, after enlarging $E$, we may assume that $C$ is geometrically irreducible. In \cite[proof of Proposition 7.4.3, Step (2), pp.~1153--1155]{pan2022fontaine}, one first chooses a two-dimensional prime $\mathfrak{q}_0\subset\mathfrak{q}$ which lies only on $C$ and for which $\rho(\mathfrak{q}_0)$ is irreducible. A lattice over the normalization of the $\mathfrak{q}$-adic completion then produces a class $B_1$ such that $C\in Z_{B_1}$. The remainder of that argument, in particular \cite[Lemma 7.4.16]{pan2022fontaine}, gives components $C_{B_0}\in Z_{B_0}$ and $C_{B_1}\in Z_{B_1}$ and a prime $\mathfrak{s}\in C_{B_0}\cap C_{B_1}$ such that $\dim R_\textnormal{aux}^\textnormal{ps}/\mathfrak{s}\geq [F:\mathbb{Q}]+2$.
		
		The component $C_{B_0}$ is pro-modular. Since the reducible locus of $R_\textnormal{aux}^\textnormal{ps}$ has dimension at most $2$ by Proposition \ref{red-pseudo}(1), Lemma \ref{Jacobon}, applied to $R_\textnormal{aux}^\textnormal{ps}/\mathfrak{s}$, gives a one-dimensional irreducible prime $\mathfrak{r}$ containing $\mathfrak{s}$. This prime is pro-modular because it lies on $C_{B_0}$. It also lies on $C_{B_1}$, so Proposition \ref{promodular to nice} implies that $C_{B_1}$ is pro-modular. The propagation claim applied to $B_1$ then shows that every component in $Z_{B_1}$ is pro-modular. In particular, $C$ is pro-modular.
		
		In all three cases, we have proved that $C$ is pro-modular. Since the reducible locus has dimension at most $2$, Lemma \ref{Jacobon} gives a one-dimensional irreducible prime on $C$. It is pro-modular by Remark \ref{remark}(1), so Proposition \ref{promodular to nice} shows that $C$ contains a nice prime. Remark \ref{dim of nice} now gives $\dim C=1+2[F:\mathbb{Q}]$.
\end{proof}

\begin{rem}\label{adv}
		1) The standing assumptions of this subsection make $\bar{\chi}$ non-quadratic. Nevertheless, Proposition \ref{promodular to nice} is particularly useful when $\bar{\chi}$ is quadratic, because the dihedral condition in Definition \ref{nice}(2) can no longer be ignored. In that situation, the lower bound $[F:\mathbb{Q}]+2$ for the intersection above does not leave enough room to impose both the local trivial-lifting conditions and the non-dihedral condition by a single specialization. Our argument first finds a one-dimensional irreducible pro-modular prime in the intersection and then applies Proposition \ref{promodular to nice}; this separates the two tasks. Once the corresponding $R=\mathbb{T}$ input is available, this strategy also applies in the exceptional $p=3$ case; see \cite[Section 5]{ZHANG2024}.
		
		2) The precise analogue of Theorem \ref{MINIMAL} is recorded in Corollary \ref{special minimal} below.
		
		3) As in Remark \ref{reason}, the proof of Theorem \ref{special case} does not require the initial assumption that $1\oplus\bar{\chi}$ is modular.
\end{rem}

\begin{cor}\label{special minimal}
	Assume further that
	$$
		[F:\mathbb{Q}]\geq1+\frac{1}{2}\operatorname{dim}_{\mathbb{F}}H^1(F_{\Sigma}/F,\mathbb{F}(\bar{\chi}^{-1})).
	$$
	Then, for any non-zero element $x\in H^1(F_{\Sigma}/F,\mathbb{F}(\bar{\chi}^{-1}))$, every prime ideal of $R_x^\textnormal{aux}$ is pro-modular. Moreover, $R_x^\textnormal{aux}$ is a local complete intersection ring of Krull dimension $1+2[F:\mathbb{Q}]$.
\end{cor}

\begin{proof}
	Let $\mathfrak{P}$ be a minimal prime of $R_x^\textnormal{aux}$. As in the proof of Theorem \ref{MINIMAL}, Propositions \ref{raux} and \ref{red} show that the representation at $\mathfrak{P}$ is irreducible, and Proposition \ref{pan}(2) shows that its inverse image $\mathfrak{P}'$ in $R_\textnormal{aux}^\textnormal{ps}$ is minimal and defines a component of dimension at least $1+2[F:\mathbb{Q}]$. Theorem \ref{special case} shows that this component is pro-modular and has dimension exactly $1+2[F:\mathbb{Q}]$. The argument with the system $\{\xi_v\}$ in the first paragraph of the proof of Theorem \ref{MINIMAL} therefore proves that every prime of $R_x^\textnormal{aux}$ is pro-modular. Finally, the dimension inequalities in Proposition \ref{pan}(2), followed by the regular-sequence argument in the last paragraph of that proof, show that $R_x^\textnormal{aux}$ is a local complete intersection of the asserted dimension.
\end{proof}

\bibliography{main.bib}

@book{artin,
	author = {Emil Artin and John Tate},
	publisher = {W. A. Benjamin, Inc., New York-Amsterdam},
	title = {Class field theory},
	year = {1968}}

@misc{ZHANG2024,
	author = {Xinyao Zhang},
	howpublished = {\url{https://arxiv.org/abs/2412.06812}},
	title = {On the {F}ontaine-{M}azur conjecture for $p=3$},
	year = {2024}}

@misc{ZhaZariskiDensity,
	archiveprefix = {arXiv},
	author = {Xinyao Zhang},
	eprint = {2512.21249},
	howpublished = {\url{https://arxiv.org/abs/2512.21249}},
	primaryclass = {math.NT},
	title = {Modular points and dimensions of {E}isenstein deformation spaces},
	year = {2026}}

@article{PMIHES_2008__108__183_0,
	author = {Taylor, Richard},
	doi = {10.1007/s10240-008-0015-2},
	journal = {Publications Math\'ematiques de l'IH\'ES},
	language = {en},
	mrnumber = {2470688},
	pages = {183--239},
	publisher = {Springer-Verlag},
	title = {Automorphy for some $l$-adic lifts of automorphic mod $l$ {Galois} representations. {II}},
	url = {http://www.numdam.org/articles/10.1007/s10240-008-0015-2/},
	volume = {108},
	year = {2008},
	zbl = {1169.11021}}

@book{MR242802,
	author = {Atiyah, M. F. and Macdonald, I. G.},
	mrclass = {13.00},
	mrnumber = {242802},
	mrreviewer = {Johnny A. Johnson},
	pages = {ix+128},
	publisher = {Addison-Wesley Publishing Co., Reading, Mass.-London-Don Mills, Ont.},
	title = {Introduction to commutative algebra},
	url = {https://mathscinet-ams-org.utokyo.idm.oclc.org/mathscinet-getitem?mr=242802},
	year = {1969}}

@misc{stacks-project,
	author = {The {Stacks project authors}},
	howpublished = {\url{https://stacks.math.columbia.edu}},
	title = {The Stacks project},
	year = {2023}}

@book{eisenbud2013commutative,
	author = {Eisenbud, David},
	publisher = {Springer Science \& Business Media},
	title = {Commutative algebra: with a view toward algebraic geometry},
	volume = {150},
	year = {2013}}

@article{allen2019automorphic,
	author = {Allen, Patrick B.},
	journal = {American Journal of Mathematics},
	number = {1},
	pages = {119--167},
	publisher = {Johns Hopkins University Press},
	title = {On automorphic points in polarized deformation rings},
	volume = {141},
	year = {2019}}

@article{bockle2001density,
	author = {B{\"o}ckle, Gebhard},
	journal = {American Journal of Mathematics},
	number = {5},
	pages = {985--1007},
	publisher = {Johns Hopkins University Press},
	title = {On the density of modular points in universal deformation spaces},
	volume = {123},
	year = {2001}}

@incollection{gouveaMazur1998density,
	author = {Gouv{\^e}a, Fernando Q. and Mazur, Barry},
	booktitle = {Computational Perspectives on Number Theory},
	pages = {127--142},
	publisher = {American Mathematical Society},
	series = {AMS/IP Studies in Advanced Mathematics},
	title = {On the density of modular representations},
	volume = {7},
	year = {1998}}

@article{pavskunas2013image,
	author = {Pa{\v{s}}k{\=u}nas, Vytautas},
	journal = {Publications math{\'e}matiques de l'IH{\'E}S},
	pages = {1--191},
	title = {The image of {Colmez's} Montreal functor},
	volume = {118},
	year = {2013}}

@inproceedings{mazur1989deforming,
	author = {Mazur, Barry},
	booktitle = {Galois Groups over $\mathbb{Q}$: Proceedings of a Workshop Held March 23--27, 1987},
	organization = {Springer},
	pages = {385--437},
	title = {Deforming {G}alois representations},
	year = {1989}}

@book{neukirch2013cohomology,
	author = {Neukirch, J{\"u}rgen and Schmidt, Alexander and Wingberg, Kay},
	publisher = {Springer Science \& Business Media},
	title = {Cohomology of number fields},
	volume = {323},
	year = {2013}}

@article{bellaiche2009families,
	author = {Bella{\"\i}che, Jo{\"e}l and Chenevier, Ga{\"e}tan},
	journal = {Ast{\'e}risque},
	pages = {1--314},
	title = {Families of {Galois} representations and {Selmer} groups},
	volume = {324},
	year = {2009}}

@article{deo2023density,
	author = {Deo, Shaunak V.},
	doi = {10.1093/imrn/rnad037},
	journal = {International Mathematics Research Notices},
	number = {21},
	pages = {18443--18499},
	publisher = {Oxford University Press},
	title = {On Density of Modular Points in Pseudo-Deformation Rings},
	volume = {2023},
	year = {2023}}

@article{skinner1999residually,
	author = {Skinner, Christopher M. and Wiles, Andrew J.},
	journal = {Publications Math{\'e}matiques de l'IH{\'E}S},
	pages = {5--126},
	title = {Residually reducible representations and modular forms},
	volume = {89},
	year = {1999}}

@article{pan2022fontaine,
	author = {Pan, Lue},
	journal = {Journal of the American Mathematical Society},
	number = {4},
	pages = {1031--1169},
	title = {The {Fontaine-Mazur} conjecture in the residually reducible case},
	volume = {35},
	year = {2022}}

@article{washington1978class,
	author = {Washington, Lawrence C.},
	doi = {10.1007/BF01399512},
	journal = {Inventiones Mathematicae},
	number = {1},
	pages = {87--97},
	title = {The non-$p$-part of the class number in a cyclotomic $\mathbb{Z}_p$-extension},
	volume = {49},
	year = {1978}}
\bibliographystyle{alpha}~
~

Graduate School of Mathematical Sciences, University of Tokyo, Komaba, Meguro, Tokyo 153-8914, Japan\\

\textit{Email address}: zhangxy96@g.ecc.u-tokyo.ac.jp

\end{document}